\documentclass[11pt,a4paper]{amsart}
\usepackage{dsfont}
\usepackage{mathtext}
\usepackage{amsthm,amssymb,amsmath,amsfonts, amscd,amsxtra}
\usepackage{mathtools}
\usepackage{enumerate}
\usepackage[all]{xy}
\usepackage{appendix}

\usepackage{xcolor} 

\usepackage{tikz} 

\newtheorem{Th}{Theorem}[section]
\newtheorem*{mTh}{Main Theorem}
\newtheorem{Prop}[Th]{Proposition}
\newtheorem{Lemma}[Th]{Lemma}
\newtheorem{Cor}[Th]{Corollary}

\newtheorem*{Quest*}{Question}
\newtheorem*{Th*}{Theorem}

\theoremstyle{definition} 
\newtheorem{Def}[Th]{Definition}
\newtheorem{Remark}[Th]{Remark}

\newtheorem*{Ack}{Acknowledgments}

\newcommand{\ON}{\operatorname}

\newcommand{\mult}{\operatorname{mult}}

\newcommand{\gem}{\geqslant}
\newcommand{\lem}{\leqslant}

\newcommand{\Cl}{\operatorname{Cl}}
\newcommand{\Cox}{\operatorname{Cox}}

\newcommand{\NE}{\operatorname{NE}}
\newcommand{\SL}{\operatorname{SL}}

\newcommand{\MC}{\mathbb{C}}
\newcommand{\MF}{\mathbb{F}}

\newcommand{\MP}{\mathbb{P}}
\newcommand{\MQ}{\mathbb{Q}}
\newcommand{\MZ}{\mathbb{Z}}

\newcommand{\CH}{\mathcal{H}}
\newcommand{\CM}{\mathcal{M}}

\title{birational rigidity of orbifold degree 2 del Pezzo fibrations}
\date{}
\author{Hamid Abban \and Igor Krylov}
\address{Department of Mathematical Sciences, Loughborough University, LE11 3TU, UK}
\email{h.abban@lboro.ac.uk}
\address{Korea Institute for Advanced Study, 85 Hoegiro, Dongdaemun-gu, Seoul 02455, Republic of Korea}
\email{IKrylov@kias.re.kr}



\begin{document}

\begin{abstract}
Varieties fibered into del Pezzo surfaces form a class of possible outputs of the minimal model program. It is known that del Pezzo fibrations of degrees $1$ and $2$ over the projective line  with smooth total space satisfying the so-called $K^2$-condition are birationally rigid: their Mori fibre space structure is unique. This implies that they are not birational to any Fano varieties, conic bundles or other del Pezzo fibrations. In particular, they are irrational. The families of del Pezzo fibrations with smooth total space of degree $2$ are rather special, as for ``most'' families a general del Pezzo fibration has the simplest orbifold singularities. We prove that orbifold del Pezzo fibrations of degree $2$ over the projective line satisfying explicit generality conditions as well as a generalised $K^2$-condition are birationally rigid.
\end{abstract}

\maketitle

\section{Introduction}

Birational classification of complex algebraic varieties is a central research area in algebraic geometry. Given an algebraic variety, one can produce a somewhat simpler birational model of it by first taking a resolution of singularities and then running the minimal model program (MMP). We work in dimension $3$ and over the field of complex numbers, where both these theories are settled. The result of this procedure is either a Mori fibre space or a minimal model, depending on whether the initial variety was uniruled or not. We are interested in explicit classification of Mori fibre spaces, that is the study or birational relations among Mori fibre spaces as end points of the MMP. A Mori fibre space can be a unique product of the MMP, so-called {\it Birationally Rigid}, can have a few, or infinitely many birational models (see Definition \ref{RigDef} for a precise definition of rigidity). For example, a smooth quartic in $\MP^4$ is known to be birationally rigid \cite{Quartic} while a quartic  with a single $cA_2$ singular point has precisely $2$ birational models \cite{Corti-Mella}. On the other hand, the projective space $\MP^3$ is birational to any Fano variety $V_{22}$, whose moduli contain an uncountable set (see for example \cite{Prokh-Kuz}).

\noindent A Mori fibre space is a normal projective variety $X$ together with a morphism $\pi\colon X\rightarrow S$ with connected fibres to a lower dimensional variety $S$, where
\begin{enumerate}[$\bullet$]
\item $X$ is $\MQ$-factorial with terminal singularities,
\item $-K_X$ is $\pi$-ample, and
\item the relative Picard number $\rho_{X/S}$ is $1$.
\end{enumerate}

Based on the dimension of $S$, either $X$ is a conic bundle over the surface $S$, or it is a fibration of del Pezzo surfaces over a curve $S$, or it is a Fano 3-fold when $S$ is a point.

\begin{Def}
Let $\pi_X \colon X \to S$ and $\pi_Y \colon Y \to Z$ be Mori fibre spaces.
A birational map $\chi \colon X \dasharrow Y$ is called \emph{square} if it fits into a commutative diagram
\begin{displaymath}
\xymatrix
{ 
	X\ar@{-->}[r]^\chi \ar[d]_{\pi_X} & \ar[d]^{\pi_Y} Y \\
	S\ar@{-->}[r]^g & Z
}
\end{displaymath}
where $g$ is birational and, in addition, the induced map on the generic fibres $\chi_\eta \colon X_\eta \dasharrow Y_\eta$ is an isomorphism.
In this case we say that $X/S$ and $Y/Z$ are square-birational.
\end{Def}

\begin{Def}[{\cite[Definition~1.2]{Corti}}]\label{RigDef}
\noindent A Mori fibre space $\pi_X\colon X\rightarrow S$ is said to be {\it birationally rigid} 
if existence of a birational map $\chi\colon X\dashrightarrow Y$ to a Mori fibre space $\pi_Y\colon Y\to Z$ implies that there exists a birational self-map $\varphi\colon X\dashrightarrow X$ such that the birational map $\chi \circ\varphi$ is square.
\end{Def}

Birational rigidity of conic bundles has been studied extensively, see for example \cite{Sar-CB}. The reader is encouraged to consult \cite[\S~4]{Corti} for an overview and some interesting related questions on birational geometry of conic bundles. Comparatively, the question of (stable) non-rationality of conic bundles has recently had more development \cite{AO, BB, HKT}. Birational rigidity and non-rationality of Fano 3-folds has progressed in the past decades but is yet to be completed; see \cite{HT, Okada} for rationality and \cite{AO-BR} for birational rigidity and investigate the references therein for further details. The focus of this paper is on birational rigidity of 3-fold del Pezzo fibrations.

del Pezzo fibrations form ten classes, according to their degrees $1\leq d\leq 9$, there are two classes in degree 8. If the degree is $5$ or higher then the fibration is rational over the base \cite{CT, Manin}. Degree $4$ fibrations admit a conic bundle structure \cite{Alekseev}, hence they are not birationally rigid. Their stable rationality was recently studied in \cite{HT}.

A smooth degree $2$ del Pezzo surface can be defined as the zeros of a quartic form in $\MP(1,1,1,2)$, and a general such form defines a smooth del Pezzo sufrace of degree $2$. However, when defined over a base curve this will typically have some orbifold singularities of type $\frac{1}{2}(1,1,1)$ enforced by the zeros of the coefficient of the quadratic monomial in the defining polynomial of the 3-fold; see \cite[\S1.]{YPG} for notation and explanation of this type of singularities. This coefficient degree imposes a discrete invariant that splits the moduli space of such fibrations into an infinite set of families. As a result, whenever we use the term general we mean general after fixing a family. In low degrees most results on birational rigidity concentrate on smooth models, for example stable rationality of very general del Pezzo fibrations in low degrees with smooth total space was recently studied in \cite{KO}. Smoothness of the total space is a strong restriction in degrees $1$ and $2$ as they are satisfied in very few families. However, a general cubic surface (degree $3$ del Pezzo) fibration has smooth total space, which is no longer true in degrees $1$ and $2$. 

The following is the main result in the literature for birational rigidity of del Pezzo fibrations over $\MP^1$ with smooth total space, which was later improved slightly by Grinenko \cite{Grin2, Grin1, Grin3} and Sobolev \cite{Sobolev}, also with the smoothness condition.

\begin{Th}[Pukhlikov, {\cite[Theorem~2.1]{Pukh123}}]\label{Pukh} Let $\pi:X \to \MP^1$ be a del Pezzo fibration of degree $1$, $2$, or $3$, with smooth $X$, and assume generality in degree $3$. If $K_X^2$ is not in the interior of the Mori cone of effective $1$-cycles $\overline{\NE}(X)^o$, then $X$ is birationally rigid.
\end{Th}

This theorem will be the centre of attention in this article. The only birational rigidity statement for del Pezzo fibrations of low degree with quotient singularities is proven in \cite{Kr16}, which proves Pukhlikov's theorem above for special families of fibrations with such singularities. We will generalise these results for the general case with orbifold families of del Pezzo fibrations in degree $2$ in the following precise setting. There is a fundamental difference in the geometry of the models in the general case, which makes the proof more complicated: the singular points produce elementary Sarkisov links to other models square-birational to $X$, and that the construction of the staircase (explained below) for smooth points near the singular point is more complex.

\begin{mTh} \label{Th1}
Let $\pi\colon X \to \MP^1$ be a del Pezzo fibration of degree $2$. Suppose $X$ is a general quasi-smooth hypersurface in a $\MP(1,1,1,2)$-bundle over $\MP^1$ satisfying the $K^2$-total condition. Then $X$ is birationally rigid.
\end{mTh}

\paragraph{\bf Conditions in the Main Theorem} Here, we spell out assumptions in the theorem above, and explain their necessity.

\noindent {\it Quasi-smoothness}: quasi-smooth means that no singularities come from the defining equation of $X$, hence all singularities are indeed of type $\frac{1}{2}(1,1,1)$ and inherited from the ambient toric variety, see Remark\,\ref{qsmooth} for a precise definition. This condition cannot be removed as there exists an example of a non-birationally rigid degree 2 del Pezzo fibration satisfying other conditions that has a non-toric singular point \cite{Hamid-example}.

\noindent {\it The $K^2$-total condition}: The $K^2$-condition is precisely as in Theorem~\ref{Pukh}, that is $K_X^2\notin\overline{\NE}(X)^o$. 
For every singular point $Q$ of $X$ there is a Sarkisov link starting by blowing up $X$ at $Q$ and resulting in a new quasi-smooth model of $X$, that is square birational. This manoeuvre is explicitly described in the Section~\ref{Involution}.
Let $N$ be the number of singularities of $X$ and denote the singular points by $Q_i$ for $i \in \{1,\dots,N \}$.
Denoting by $X_I$ the model acquired by combining the elementary links corresponding to $Q_i$, $i \in I$, $I\subset \{1,\dots,N \}$, we say $X$ satisfies the $K^2$-total condition if for every $I \subset \{1,\dots,N\}$ the model $X_I$ satisfies the $K^2$-condition.
We are convinced that $K^2$-condition on $X$ may be enough, and the totality assumption is redundant. In Section~\ref{Involution}, we give an explicit recipe for constructing all $X_I$ models from $X$. Given $X$ embedded in a $\mathbb{P}(1,1,1,2)$-bundle it is rather easy to check this condition using the recipe. 
Note that in \cite[Corollary~6.4]{Shokurov-Choi} it is stated that $X$ is birationally rigid if the $K^2$-condition holds for any del Pezzo fibration $X^\prime \to \MP^1$ such that there is a square birational map $g: X \dasharrow X^\prime$.
Of course this result is impossible to apply in practice and usually one works with only $X$ in order to show birational rigidity.
However, in our situation we have explicit descriptions of all necessary models (for this check to be carried out) as described in Section \ref{Involution}.

\noindent {\it Generality}:
For each singular point, let the fibre containing the point, that is defined by a quartic $F$ in $\MP(1_x,1_y,1_z,2_w)$, be given by
$$w q(x,y,z) + r(x,y,z)=0,$$
where $q$ and $r$ are homogeneous polynomials of degrees $2$ and $4$ respectively.
We require that the intersection $q(x,y,z) = r(x,y,z)=0$ on $\MP^2$ consists of $8$ distinct points, for all singular points of $X$.
Let $\sigma: \widetilde{F} \to F$ be the blow up of $F$ at the point $(0:0:0:1)$.
We also require that $\widetilde{F}$ is smooth. This generality condition is necessary in our computations, as otherwise there are several arrangements of singularities and curves of low degree on $F$ to be considered in the ``exclusion'' process of the proof. In a few cases we have checked, this condition can be dropped. Upon improving the techniques we believe that the generality condition can be dropped altogether. Note that the models considered in \cite{Kr16} are special cases in our setting as for them $q=0$. They admit no Sarkisov link and in a sense are ``more rigid''.

\subsection*{Method of the proof} Birational rigidity is often proved via the method of maximal singularities~\cite{Pukh-MM}. Roughly speaking, the method goes as follows: assume there is a birational map $\chi \colon X\dashrightarrow Y$, where $Y\to Z$ is another Mori fibre space. Then consider $\mathcal{H}$, the transform of a very ample complete linear system $\mathcal{H}'$ on $Y$. Essentially, $\mathcal{H}$ is mobile and there exists a rational number $n>0$ and a divisor $A$ pulled back from the base of the fibration such that $\mathcal{H}\subset\mid -nK_X+A\mid$. It follows that the pair $(X,\frac{1}{n}\mathcal{H})$ is not canonical (this is standard, see for example the discussion after Theorem 1.9 in \cite{Corti}). This implies that there exists a valuation $E$ with centre $C\subset X$ such that
\[m_E(\mathcal{H})>na_E(K_X),\]
where $a_E(K_X)$ is the discrepancy of $E$ with respect to $K_X$ and $m_E$ is the multiplicity of $\mathcal{H}$ along $C$. This inequality together with the geometry of $X$ is then used to exclude many centres (curves and points on $X$) to satisfy this conditions by concluding various contradictions. Pukhlikov in \cite[Section~6]{Pukh123} studied multiplicities on towers of blow ups of centre on $X$ and introduced the {\it construction of the staircase} to achieve the desired contradiction. Corti in \cite[\S~5]{Corti} refines Pukhlikov's methods for the situations where the staircase technique is not necessary to reduce the computations without the tower by introducing a new inequality. We use a combination of these two techniques in Sections~\ref{super-max} and \ref{staircase} to exclude all smooth centres; we will then have to use the staircases together with Corti's inequality at various stages to obtain our results. We then show that the singular points produce Sarkisov links to other models $X_I$ of the del Pezzo fibration $\pi: X \to \MP^1$. We show that there is at least one $X_I$ with centres only at the smooth points, which allows a combination of techniques of Corti and Pukhlikov to be efficiently used.

\subsection{Models}\label{models-subs} Suppose $X\to\MP^1$ is a del Pezzo fibration of degree $2$, and view $X$ as a hypersurface of bi-degree $(d,4)$ in a toric variety $T$ with Picard group $\MZ^2$, where Cox ring of $T$ is given by the following data:
\begin{enumerate}[(i)]
\item the homogeneous coordinate ring of $T$ is $\Cox(T)=\MC[u,v,x,y,z,t]$,
\item with the irrelevant ideal $I=(u,v)\cap(x,y,z,t)$ and
\item the grading given by the columns of the matrix 
$$
\left(\begin{array}{ccccccc}
u&v&x & y & z & w   \\
1&1&\alpha&\beta&\gamma&\delta  \\
0&0&1 & 1 & 1 & 2
\end{array}\right),
$$
where $\alpha,\beta,\gamma,\delta$ are integers. As this description is invariant under an action of $\SL(2,\MZ)$ on the weight matrix, we can rescale to the following
$$
\left(\begin{array}{ccccccc}
u&v&x & y & z & w   \\
1&1&0&a&b&c  \\
0&0&1 & 1 & 1 & 2
\end{array}\right),
$$
where $0\leq a\leq b$ are positive integers and $c\in\MZ$. We denote this toric variety by $\MP(1,1,1,2)_{(0,a,b,c)}$, and say $X$, the hypersurface, has bi-degree $(l,4)$ in the new coordinate weights. 
\end{enumerate}

\begin{Remark} Note that, once $a,b$ and $c$ are fixed not all values of $l$ can define a suitable hypersurface. For example, if $\frac{l}{4}<a,b,\frac{c}{2}$, where $a,b,c$ are positive integers, then the equation of $X$ will be divisible by $x$, making $X$ reducible. There are various restrictions that one must take into account to have a suitable del Pezzo fibration. We refer to Section\,5 in \cite{Ahm} for various cases that are not suitable. For us, the set of discrete invariants that make sense in the Main Theorem are stated in Proposition\,\ref{PrK2} and in the Appendix.\end{Remark}

\begin{Def}\label{qsmooth} Assume the discrete invariants define a suitable del Pezzo fibration. Then quasi-smoothness means that the singular locus of the affine cover of $X$ in $\mathbb{C}^6$ is included entirely in the irrelevant locus defined by $I$.\end{Def}

\begin{Remark} A general member of a family of hypersurfaces of bi-degree $(4,l)$ in $\MP(1,1,1,2)_{(0,a,b,c)}$ is smooth only if $l = 2c$. Note that this is because such hypersurface does not intersect the singular curve of the ambient toric variety. This confirms how restrictive the smoothness condition is for these varieties.
\end{Remark}

For more analysis of this construction, and a complete list of those models that admit a Sarkisov link of Type III or IV, we refer the reader to \cite{Ahm}.

The following results describe some $X\subset T$, as above, which satisfy $K^2$-total condition in the Main Theorem.

\begin{Prop} \label{PrK2}
Let $X$ be a hypersurface of bi-degree $(4,\ell)$ in a $\MP_{0,a,b,c}(1,1,1,2)$-bundle over $\MP^1$.
Suppose 
\begin{align*}
7\ell/2 > 2a + 4b + 3c + 8
\end{align*}
then $X$ satisfies $K^2$-total condition.
\end{Prop}

\begin{Cor} \label{K^2-families}
Suppose $c \gem 2b$ and $X$ does not satisfy the $K^2$-total condition.
Then $X$ belongs to one of the $20$ families with $a,b,c,\ell$ from the following table. In particular, if $c \gem 2b$, $X$ satisfies the generality conditions, and $X$ is not birationally rigid, then $X$ belongs to one of these $20$ families.
\begin{center}
\begin{tabular}{|l|l|l|l|}
\hline
 $a$ & $b$ & $c$ & $\ell$ \\\hline
 $0$ & $0$ & $0$ & $0,1,2$ \\\hline
 $0$ & $0$ & $1$ & $2,3$ \\\hline
 $0$ & $0$ & $2$ & $4$ \\\hline
 $0$ & $1$ & $2$ & $4,5$ \\\hline
 $0$ & $1$ & $3$ & $6$ \\\hline
 $0$ & $2$ & $4$ & $8$ \\\hline
 $1$ & $1$ & $2$ & $4,5$ \\\hline
 $1$ & $1$ & $3$ & $6$ \\\hline
 $1$ & $2$ & $4$ & $8$ \\\hline
 $2$ & $2$ & $4$ & $8,9$ \\\hline
 $2$ & $2$ & $5$ & $10$ \\\hline
 $2$ & $3$ & $6$ & $12$ \\\hline
 $3$ & $3$ & $6$ & $12$ \\\hline
 $4$ & $4$ & $8$ & $16$ \\\hline
\end{tabular}
\end{center}
\end{Cor}

\paragraph{\bf Convention} We denote numerical equivalence by $\equiv$ and use $\sim_\MQ$ for linear equivalence of $\MQ$-divisors. All varieties are algebraic, normal and defined over $\MC$ unless stated otherwise.
For a rational map $\chi: Y \dasharrow X$ we denote the proper transform on $Y$ of an object on $X$ as $\chi^{-1}_*(A)$.

\begin{Ack}
A substantial part of this work was carried out while the second author was at the Max Planck Institute for Mathematics in Bonn. We would like to thank that institute for great hospitality and stimulating research environment. We also would like to thank Dr.\ Okada Takuzo for useful communications and for pointing out some inaccuracies in an earlier version. The first author is partially supported by the EPSRC grant EP/T015896/1. The second author is supported by KIAS Individual Grant n.\ MG069801.
\end{Ack}

\section{Preliminary Results: Singularities of pairs}

In this section, we recall the notion of singularities of pairs.
We also state some results which let us find relations between multiplicities and singularities.

Let $X$ be an algebraic variety, possibly non-projective and singular. Let $E$ be a prime divisor on $X$. 
Then it is easy to associate a discrete valuation $\nu_E$ of $\MC(X)$ corresponding to $E$ by
 $$\nu_E(f)=\mult_E (f)\text{ for }f \in \MC(X).$$

\begin{Def}[\cite{Valuations}]
Let $\varphi:\widetilde{X}\to X$ be a projective birational morphism and let $\nu$ be a discrete valuation of $\MC(X)$. 
We  say that a triple $(\widetilde{X},\varphi,E)$ is a \emph{realization} of the discrete valuation $\nu$ if $E$ is a prime divisor on $\widetilde{X}$ and $\nu_E=\nu$.
Then $\varphi(E)$ is called the \emph{centre} of the discrete valuation $\nu_E$ on $X$.
\end{Def}

Note that if $X$ is projective, then every discrete valuation of the field $\MC(X)$ admits a centre on $X$, which does not depend on the realization.

\begin{Def}
Let $D$ be a $\MQ$-divisor on $X$. 
The \emph{multiplicity} of a discrete valuation $\nu$ at $D$ is 
\begin{align*}
\nu(D)=\mult_E \varphi^*(D) \in \MQ
\end{align*}
for some realization $(\widetilde{X},\varphi,E)$ of $\nu$. If the centre of $\nu$ on $X$ is of codimension $\geqslant 2$, then we can write
\begin{align*}
\varphi^*(D)=\varphi^{-1}_*(D)+\nu(D)E+\sum a_i E_i,
\end{align*}
where $E_i$ are the exceptional divisors of $\varphi$ that differ from $E$ and $a_i\in\MQ$. 
It is important to note that the multiplicity does not depend on the choice of the realization.
\end{Def}

\begin{Def}
Let $D$ be a $\MQ$-divisor on $X$ such that $K_X+D$ is $\MQ$-Cartier. 
Let $\pi:\widetilde{X}\to X$ be a birational morphism and let $\widetilde{D}=\pi^{-1}_*(D)$ be the proper transform of $D$. 
Then
\begin{align*}
K_{\widetilde{X}}+\widetilde{D}\sim\pi^*(K_X+D)+\sum_E a(E,X,D)E,
\end{align*}
where $E$ runs through all distinct exceptional divisors of $\pi$ on $\widetilde{X}$ and $a(E,X,D)$ are rational numbers. 
The number $a(E,X,D)$ \big(=$a(\nu_E,X,D)$\big) is called the \emph{discrepancy} of the divisor $E$ (discrete valuation $\nu_E$) with respect to the pair $(X,D)$.
\end{Def}

\begin{Def}
Let $D$ be a divisor on $X$. We say that the pair $(X,D)$ is \emph{terminal} (respectively \emph{canonical}) \emph{at a discrete valuation} $\nu$ with a centre on $X$ if $a(E,X,D)>0$ (respectively $a(E,X,D)\geqslant0$) for some realization $(\widetilde{X},\varphi,E)$ of $\nu$. 
We say that the pair $(X,D)$ is terminal (respectively canonical) \emph{at a subvariety} $Z$ if it is terminal (respectively canonical) at every discrete valuation $\nu$ on $K(X)$ such that a centre of $\nu$ on $X$ is $Z$. 
We say that the pair $(X,D)$ is terminal (respectively canonical) if it is terminal (respectively canonical) at every subvariety of codimension $\geqslant 2$. 
If $D=0$, we simply say that $X$ has only terminal (respectively canonical) singularities.
\end{Def}

\begin{Def}
Let $\mathcal{M}$ be a linear system, not necessarily mobile, on $X$. 
We say that the pair $(X,\lambda\mathcal{M})$ is terminal (respectively canonical) if for every subvariety $Z$ of codimension $\geqslant 2$ and for a general $D\in \mathcal{M}$ the pair $(X,\lambda D)$ is terminal (respectively canonical) at $Z$. 
\end{Def}

\begin{Remark} \label{RemLogPull}
Consider the pair $(X,\mathcal{M})$. Let $f:Y\to X$ be a projective birational morphism and $E_i$ the exceptional divisors of $f$, and $\widetilde{\mathcal{M}}$ the proper transform of $\mathcal{M}$ on $Y$. Then
\begin{align*}
K_Y+\widetilde{\mathcal{M}}-\sum a(E_i,X,\mathcal{M}) E_i\sim f^*(K_X+\mathcal{M}).
\end{align*}
The pair
\begin{align*}
\big(Y,\widetilde{\mathcal{M}}-\sum a(E_i,X,\mathcal{M}) E_i\big)
\end{align*}
is called the \emph{log pullback} of the pair $(X,\mathcal{M})$.
It follows from the definition that the log pullback of the pair has the same singularities as the pair.
Note that we view $\widetilde{\mathcal{M}}-\sum a(E_i,X,\mathcal{M}) E_i$ as a multiple of a linear system with fixed components.
\end{Remark}

\begin{Lemma}[{\cite[Theorem~1.8]{inequalities}}] \label{LemVanya}
Let $\mathcal{M}$ be a mobile linear system on $\MC^2$. 
Let $C$ be a curve passing through the origin. 
Suppose the pair $(\MC^2,\frac{1}{n}\mathcal{M}-\alpha C)$ is not canonical at $0$ then for general $D_1,D_2 \in \mathcal{M}$
\begin{align*}
\mult_0 D_1\cdot D_2 > 4n^2\alpha.
\end{align*}
\end{Lemma}

\begin{Prop}[{Corti's inequality, \cite[Theorem~3.12]{Corti}}] \label{CortiIneq}
Let $\sum F_i \subset \MC^3$ be a reduced surface. 
Let $\mathcal{M}$ be a mobile linear system on $\MC^3$ and let $Z=D_1\cdot D_2$ be the intersection of general divisors $D_1,D_2\in \mathcal{M}$. 
Write
$Z=Z_{h}+\sum Z_i$, where the support of $Z_i$ is contained in $F_i$ and $Z_h$ intersects $\sum F_i$ properly. 
Let $\gamma_i > 0$ be rational numbers and such that the pair $(\MC^3,\frac{1}{n}\mathcal{M}-\sum \gamma_i F_i)$ is not canonical. 
Then there are positive rational numbers $0 < t_i \leqslant 1$ such that
\begin{align*}
\mult_0 Z_h + \sum t_i \mult_0 Z_i > 4n^2(1+\sum \gamma_i t_i\mult_0 F_i).
\end{align*}
\end{Prop}

\begin{Remark}
Decomposition of $Z=Z_h+\sum Z_i$ is not unique, but the inequality holds for any choice of the decomposition. 
Note that it is not a requirement that $\sum F_i$ is a normal crossing divisor, nor do we ask the surfaces $F_i$ to be smooth.

\end{Remark}

\section{Rigidity of del Pezzo fibrations} \label{SecRig}

\subsection{Noether-Fano method for del Pezzo fibrations}
Let us first describe the framework of proving birational rigidity for del Pezzo fibrations.
Let $\pi: X \to \MP^1$ be a del Pezzo fibration (satisfying $K^2$-condition) and let $\chi\colon X \dasharrow Y$ be a birational map to a variety admitting a Mori fibre space $\pi_Y\colon Y\to Z$.
Let $H$ be a very ample divisor on $Y$ and $\CM = \chi^{-1}_* (\big| H \big|)$; We say that $\CM$ is a \emph{mobile linear system associated to} $\chi$.
There are non-negative numbers $n \in \MZ$ and $l \in \MQ$ such that 
$\CM \subset \big| -n K_X + n l F \big|$, where $F$ is a fibre of $\pi$.

The core of the Noether-Fano method is the following. 
Suppose $\chi$ is not an isomorphism, then the pair $(X, \frac{1}{n} \CM)$ is not canonical.
There are several possibilities for the centres of singularities of this pair:
a singular point of $X$, a curve passing through a singular point of $X$, any other curve of $X$, and a nonsingular point of $X$.
To prove $X$ is birationally rigid, one uses the geometric properties of $X$ to exclude a centre or untwist $\chi$ using that centre. Excluding goes by proving that the pair is canonical at the centre, and untwisting composes $\chi$ with some square birational map $\varphi\colon X\dashrightarrow X'$, where $\chi\circ\varphi^{-1}$ has lower Sarkisov degree than $\chi$ (see \cite[\S2.1]{Corti} for the definition). Thus the mobile linear system associated to $\chi \circ \varphi^{-1}$ is less complicated, and the model $X$ is replaced by $X'$. 

We do this analysis for orbifold del Pezzo fibrations of degree $2$.
Let $X$ be a quasi-smooth hypersurface in a $\MP(1,1,1,2)$-bundle over $\MP^1$.
Suppose also that $\pi: X \to \MP^1$ is a del Pezzo fibration of degree $2$ and that $X$ satisfies the $K^2$-condition.

\subsection{Step 1: curves not passing through singular locus of $X$}
This part is well known by the work of Pukhlikov \cite{Pukh123}.

\begin{Prop}[{\cite[\S3, Cases~2~and~4]{Pukh123}}] \label{Exc-Curves}
Let $C$ be a curve on $X$ which is not a section of $\pi$.
Suppose also that $C$ does not pass through singular points of $X$.
Then the pair $(X,\frac{1}{n}\mathcal{M})$ is canonical at $C$.
\end{Prop}

\begin{Prop}[{\cite[\S3, Case~2]{Pukh123}}] \label{Untwist-Curves}
Let $C$ be a section of $\pi$ which does not pass through a singular point of $X$.
Then either the pair $(X,\frac{1}{n}\mathcal{M})$ is canonical at $C$ or there exists a birational involution $\varphi_C$ such that
\begin{itemize}
	\item there are numbers $n^\prime<n$ and $l^\prime$ such that $\varphi_C(\mathcal{M})\subset \big| -n^\prime K_X +n^\prime l^\prime F \big|$, and
	\item the pair $\big( X, \frac{1}{n^\prime} (\varphi_C)(\mathcal{M}) \big)$ is canonical at $C$.
\end{itemize}
\end{Prop}

\begin{Cor}[{\cite[Section~3]{Pukh123}}] \label{No-Curves}
There is a birational map $\varphi \colon X \dasharrow X$ such that for the linear system $\varphi(\CM) \subset \big| -n^\prime K_X + n^\prime l^\prime \big|$ the pair $(X,\frac{1}{n^\prime} \varphi(\CM))$ is canonical at curves in $X \setminus \ON{Sing} X$ and $\varphi = \varphi_{C_1} \circ \dots \circ \varphi_{C_m}$, where $C_i$ are sections of $X$.
\end{Cor}
\begin{proof}
Apply Proposition \ref{Exc-Curves} and then apply Proposition \ref{Untwist-Curves} as many times as necessary.
The process terminates since for every application of Proposition \ref{Untwist-Curves} we have $n^\prime < n$.
\end{proof}

In what follows, we write $\chi$ instead $\chi \circ \varphi^{-1}$ to simplify notation, and safely assume it is canonical at curves in the smooth locus. In the remainder of this article, we show that points and curves passing through singular points are either excluded or untwisted, as explained in the following steps.

\subsection{Step 2: singular points of $X$ and curves passing through them}

Singular points of $X$ are cyclic quotient singularities. By the famous result of Kawamata on extremal extractions from such points, if the centre of non-canonicity contains the singular point then it is only that point and cannot be a curve (see Section \ref{Involution}). So it remains to study the singular points as centres.  In Section \ref{Involution} we construct birational maps to other models of del Pezzo fibration $\varphi_I: X \dasharrow X_I$, with the diagram
\begin{displaymath}
\xymatrix
{ 
	X\ar@{-->}[r]^{\varphi_I} \ar[d]_{\pi} & \ar[d]^{\pi_I} X_I \\
	\MP^1 \ar@{=}[r]^g & \MP^1,
}
\end{displaymath}
being commutative.
Let $\CM_I \subset \big| -n K_{X_I} + n l_I F \big| $ be the linear system corresponding to the map $\chi \circ \varphi_I$.
We prove that the pair $(X_I, \frac{1}{n} \CM_I)$ may only be non-canonical at nonsingular points of $X_I$ for appropriate choice of $I$.
Also, we show that by construction $X_I$ satisfies the generality conditions in the Main Theorem. This untwists $\chi$ for singular points.
From that point on, with an abuse of notation, instead of $\chi$ we work with $\chi \circ \varphi_I$ and call it $\chi$.

\subsection{Step 3: non-singular points of $X$}
After Step 1 and 2 we conclude that if $(X,\frac{1}{n}\CM)$ is not canonical at a subvariety $B$, then $B$ is a point and $X$ is smooth at $B$.
This may indeed occur if $\chi$ is square birational.
If we assume that $\chi$ is not square birational, then we can use Proposition \ref{supermaximal} to get a stronger requirement: $(X,\frac{1}{n}\CM - \gamma F)$ is not canonical at $B$, where $F$ is a fibre  containing $B$ and $\gamma$ is a number depending on $\CM$. Recall that there are lines in the fibres that contain singular points of $X$ which intersect $-K_X$ with $\frac{1}{2}$. If $B$ does not lie on a such line, in particular when $F$ does not contain singularities of $X$, then Corti's inequality together with Proposition \ref{supermaximal} give us the necessary contradiction to exclude $B$.
This is the approach of \cite{Corti} and we cover it in Section \ref{super-max}.
Sections \ref{staircase}, \ref{multstaircase}, and \ref{SMupstairs} are entirely devoted to the remaining case: excluding the possibility that the centre is a nonsingular point on a curve in a fibre  intersecting $-K_X$ by $\frac{1}{2}$. As the reader expects, this is the most complicated case. We combine the approach of \cite{Corti} with the ``staircase'' used for cubic fibrations in \cite[Sections~6-7]{Pukh123} to exclude these nonsingular points. 

\section{Square maps from $\frac{1}{2}(1,1,1)$-points} \label{Involution}

In this section we explicitly work out the square birational maps produced by blowing up the singular points.
Let $X$ be the hypersurface of bi-degree $(\ell,4)$ given by the equation
$$
f(u,v) w^2 + q(x,y,z;u,v) w + r(x,y,z;u,v)=0. 
$$
in the toric variety $T=\MP(1_x,1_y,1_z,2_w)_{(0,a,b,c)}/\MP^1_{u:v}$.

\subsection{Simple case: $X$ has one singular point}\label{simple-case}
First, for simplicity suppose $f(u,v)=u$.
Then the fibre $F$ given by $u = 0$ contains the singular point of $X$.
We now describe each map of the following diagram

\begin{displaymath}
\xymatrix
{ 
	\widetilde{U}\ar@{-->}[r]^\psi \ar[d]_{\sigma} & \ar[d]^{\bar{\sigma}} \overline{U} \\
	U\ar@{-->}[r]^\varphi & U.
}
\end{displaymath}

The variety $U$ is the open subset of $X$ given by $v \neq 0$. It can be viewed as the locus $\{f(u,1) w^2 + q(x,y,z;u,1) w + r(x,y,z;u,1)=0\}$ inside $\MP(1_x,1_y,1_z,2_w)\times\mathbb{C}^1_u$.
The map $\sigma$ is the Kawamata blow up of $U$ along the $\frac{1}{2}(1,1,1)$-point \cite{Kaw}.
We may describe $\widetilde{U}$ as a hypersurface in $\widetilde{V}$, where $\widetilde{V}$ is a quasi-projective toric variety with $\Cox \widetilde{V} = \MC[x,y,z,w,u,\bar{u}]$.
The grading of this ring is given by the matrix
$$
\left(\begin{array}{ccccccc}
u & x & y & z & w & \bar{u}     \\
0 & 1 & 1 & 1 & 2 &     0        \\
2 & 1 & 1 & 1 & 0 &   -2 
\end{array}\right)
$$
and irrelevant ideal $(u,x,y,z)\cap(w,\bar{u})\cap(x,y,z,w)$; see \cite[\S3.2]{Ahm-Cr} for a general treatment of cox rings of blowups of rank 2 toric varieties. Note that we are looking at the open set $\{v\neq 0\}$, hence the rank drops to 2 and the ideal simplifies. We may well-form (\cite[\S2]{Ahm-Cr}) the matrix above and rewrite it as
$$
\left(\begin{array}{ccccccc}
u & x & y & z & w &  \bar{u}     \\
0  & 1 & 1 & 1 & 2 & 0        \\
-1 & 0 & 0 & 0 & 1 &  1 
\end{array}\right).
$$

The equation of $\widetilde{U}$ is
$$
u w^2 + q(x,y,z; \bar{u} u) w + \bar{u} r(x,y,z; \bar{u} u)=0,
$$
where $q(x,y,z; u) = q(y,x,z; u,1)$ and $r(x,y,z;u)=r(x,y,z;u,1)$.
To produce a Sarkisov link on $\widetilde{U}$, we run a 2-ray game on the ambient space ($\widetilde{V}$ to begin with) and restrict it to the 3-fold \cite{AZ}. Note that the 2-ray game corresponds to the variation of GIT, inevitably changing $(u,x,y,z)\cap(w,\bar{u})$ part in the irrelevant ideal. The component $(x,y,z,w)$ in the ideal will remain unchanged, and preserves the fibres (the 2-ray game is relative over the base of the fibration). The first step of the game on $\widetilde{V}$ does not produce a 2-ray game on $\widetilde{U}$, caused by the fact that the equation of the hypersurface belongs to the irrelevant component $(w,\bar{u})$, hence we are forced to do an unprojection (preserving $\widetilde{U}$). This will embed $\widetilde{U}$ in the toric variety defined by the Cox ring $\MC[x,y,z,w,u,\bar{u}]$ graded by 
$$
\left(\begin{array}{cccccccc}
u & t & x & y & z & w &  \bar{u}     \\
0  & 2 & 1 & 1 & 1 & 2 & 0        \\
-1 & -1 & 0 & 0 & 0 & 1 &  1 
\end{array}\right).
$$
The irrelevant ideal is $(u,t,x,y,z)\cap(w,\bar{u})\cap(x,y,z,w,t)$ and $\widetilde{U}$ is a complete intersection of two hypersurfaces:
$$
u w + q(x,y,z; \bar{u} u) = \bar{u}t 
\quad\text{and}\quad 
-wt=r(x,y,z; \bar{u} u)
$$
The toric 2-ray game now restricts to a 2-ray game on $\widetilde{U}$. 
The first step is the flop of $8$ lines, which corresponds to the $8$ solutions of $\{q(x,y,z,0)=r(x,y,z,0)=0\}\subset\MP^2$, mapping $\widetilde{U}$ to $\bar{U}$ by crossing the $(1,0)$ wall.
Recall that the solutions of $\{q(x,y,z,0)=r(x,y,z,0)=0\}$ are distinct because of the generality conditions in the Main Theorem.

The map $\bar{\sigma}$ is defined as
$$
(x,y,z,w,t,u,\bar{u}) \mapsto (u^{\frac{1}{2}}x,u^{\frac{1}{2}}y, u^{\frac{1}{2}}z, w, u t, u \bar{u}).
$$
It contracts the proper transform of $F$ into a $\frac{1}{2}(1,1,1)$-point.
Thus we see that the equations of the image of $\bar{\sigma}$ are
$$
\bar{u} t = w + q(x,y,z; \bar{u})
\quad\text{and}\quad
w t + r(x,y,z; \bar{u}) = 0.
$$
We can eliminate the variable $w$, thus we have only one equation
$$
\bar{u} t^2 - q(x,y,z;\bar{u}) w + r(x,y,z;\bar{u})=0.
$$
Clearly, the this variety is isomorphic to $U$. 
Thus the map $\varphi$ is a birational automorphism.
We may write the map $\varphi$ in coordinates as
$$
(x,y,z,w,u) \mapsto (x,y,z,-w+\frac{q(x,y,z;u)}{u};u)
$$
In particular we see that $\varphi$ is a local involution. 


\subsection{Birational models of $X$: square birational maps}\label{square-models}

Note that the description of maps and models in the simple case in Subsection~\ref{simple-case} works out for any calculations around the local ring near the point in the image of the singular point at the base curve. We use this to describe all models $X_I$.

Embed $X$ into $V$, a toric $\MP(1,1,1,2,2)$-bundle over $\MP^1$, as a complete intersection as follows. Let $N$ be the degree of $f \in \MC[u,v]$. Let the weight matrix of the toric variety $V$ be
$$
\left(\begin{array}{ccccccc}
v & u & x & y & z & w & s     \\
0 & 0 & 1 & 1 & 1 & 2 & 2        \\
1 & 1 & 0 & a & b & c & c+N 
\end{array}\right)
$$
with irrelevant ideal $(u,v)\cap(x,y,z,w,s)$, and
suppose the equations of $X$ are
\begin{align*}
&s = f(u,v) w + q(x,y,z;u,v),
&s w + r(x,y,z;u,v) = 0.
\end{align*}

Note that the variable $s$ can be eliminated altogether using the first equation, and then what we are left with is the model presented in Subsection\,\ref{models-subs}.

The quasi-smoothness of $X$ implies that $f$ has $N$ distinct zeros, hence it can be written in the form $f(u,v) = l_1(u,v) \dots l_N(u,v)$, a product of distinct linear forms.

Following the Sarkisov link described in Subsection~\ref{simple-case}, we can construct $X_{\{1\}}$, the variety obtained at the end of the Sarkisov link starting from the Kawamata blow up of the singular point in the fibre $l_1=0$:

The 3-fold $X_{\{1\}}$ is defined as a complete intersection defined by equations
$$
l_1s = l_2\dots l_N w + q(x,y,z;u,v)
\quad\text{and}\quad
s w + r(x,y,z;u,v) = 0.
$$
in the toric 5-fold $V_1$ with grading
$$
\left(\begin{array}{ccccccc}
v & u & x & y & z & w & s     \\
0 & 0 & 1 & 1 & 1 & 2 & 2        \\
1 & 1 & 0 & a & b & c+1 & c+N-1 
\end{array}\right)
$$

Note that the irrelevant ideal remains unchanged. Similarly $X_I\subset V_{\big| I \big|}$ is defined by
\begin{align*}
\prod_{i \in I} l_i s = \prod_{j \not\in I} l_j w + q(x,y,z;u,v)
\quad\text{and}\quad
s w + r(x,y,z;u,v) = 0.
\end{align*}
where $V_{\big| I \big|}$ is graded by
$$
\left(\begin{array}{ccccccc}
v & u & x & y & z & w & s     \\
0 & 0 & 1 & 1 & 1 & 2 & 2        \\
1 & 1 & 0 & a & b & c + \big|I\big| & c + N - \big|I\big|
\end{array}\right).
$$
At last we denote by $\pi_I: X_I \to \MP^1_{u,v}$ the restriction of the projection $V_{\big| I \big|} \to \MP^1_{u,v}$.

While this defines the birational maps between, say $X_I$ and $X_{I\cup\{j\}}$ for $j\notin I$, the description of this birational map as an elementary Sarkisov link is identical to the case in Subsection\,\ref{simple-case}, after a suitable change of coordinates.

\begin{Remark}
The extremal contraction $\widetilde{X} \to X$ is constructed so that the image of the exceptional locus is a terminal quotient singularity \cite[Theorem~5]{Kawamata}, and locally is the Kawamata blow up of that point. This is a unique extremal extraction according to Kawamata. Then the 2-ray game is followed by a simple Atiyah flop, in particular it does not change the singularities of $\widetilde{X}$. In short, the links we described here are unique elementary Sarksiov links starting at $\frac{1}{2}(1,1,1)$-points and, as it can be seen from the coordinate description of the birational map, the links we constructed are birational involutions in analytic neighbourhood of the central fibre.
In particular it follows that $\pi_I: X_I \to \MP^1_{u,v}$ is a del Pezzo fibration.
\end{Remark}

\begin{Remark} It is clear from the symmetric equation of $X_I$ described above that $X_I\cong X_{J}$, where $J=\{1,\dots,N\}\backslash I$. And, in particular, $X_{\{1,\dots,N\}}\cong X$.\end{Remark}

\subsection{Canonicity at $\frac{1}{2}(1,1,1)$-point}
Let $\CM \subset \big| -n K_X + l F \big|$ be the mobile linear system associated to a birational map $\chi \colon X \dasharrow Y$ to a total space $Y$ of a Mori fibre  space.
Let $\mathcal{M}_I$ be the proper transform of $\mathcal{M}$ on $X_I$.
We now prove that there is an $I \subset \{1,\dots,N\}$ for which the linear system $\mathcal{M}_I$ is canonical at $\frac{1}{2}(1,1,1)$-points and at curves passing through them.

\begin{Prop}[{\cite[Theorem~5]{Kawamata}}] \label{PropKawamata}
Let $g:\widetilde{X}\to X$ be the blow up at a $\frac{1}{2}(1,1,1)$-point $Q$, let $D$ be an effective $\MQ$-divisor on $X$, and let $E_Q$ be the exceptional divisor of $g$. 
Then a pair $(X,D)$ is canonical at $Q$ if and only if it is canonical at $\nu_{E_Q}$, that is $a(E_Q,X,D) \gem 0$.
\end{Prop}

\begin{Cor} \label{CorHPCurve}
Suppose the pair $(X,D)$ is not canonical at a curve $C$ passing through the $\frac{1}{2}(1,1,1)$-point $Q$. 
Then the pair $(X,D)$ is not canonical at $Q$.
\end{Cor}

Because of this corollary we only concentrate on $\frac{1}{2}(1,1,1)$ being the centre of non-canonicity.
Note that it is possible for the pair $(X,\frac{1}{n} \CM)$ to be non-canonical at a singular point of $X$. What is true, and proven here, is that there always exist a model $X_I$ pf $X$ for which the pair $(X_I,\frac{1}{n} \CM_I)$ is canonical at singular points.

\begin{Prop}
Let $\pi: X \to \MP^1$ be a del Pezzo fibration as in the Main Theorem.
Let $\CM \subset \big| -n K_X + l F \big|$ be the mobile linear system associated to a birational map $\chi \colon X \dasharrow Y$ to the total space $Y$ of a Mori fibre space. Denote singular points of $X$ by $Q_1,\dots,Q_N$. Let $I \subset \{1,\dots,N\}$ be the maximal subset of indices such that the pair $(X, \frac{1}{n} \CM)$ is not canonical at $Q_i$ for $i \in I$. Let $\CM_I$ be the proper transform of $\CM$ on $X_I$. Then the pair $(X_I, \frac{1}{n} \CM_I)$ is canonical at all  singular points of $X_I$.
\end{Prop}
\begin{proof}
We prove it for $I = \{1\}$, the rest follows by induction. Suppose $\varphi\colon X\dashrightarrow X'=X_I$ is the birational transform described earlier starting from the blowup of the $\frac{1}{2}$ point $Q_1$, which factors as

\begin{displaymath}
\xymatrix
{ 
	\widetilde{X}\ar@{-->}[r]^{\psi} \ar[d]_{f} & \ar[d]^{f_I} \widetilde{X_I} \\
	X\ar@{-->}[r]^{\varphi} & X_I
	}
\end{displaymath}
with $f$ and $f_I$ being the Kawamata blowups of the singular points and $\psi$ the flopping map between $\widetilde{X}$ and $\widetilde{X_I}$. 
We have the following relations
\[\CM\sim -n K_X + l F\quad\text{and}\quad \CM_I\sim-n K_{X_I} + l_I F_I,\]
\[\widetilde{\CM}=f^*\CM-\nu E\quad\text{and}\quad \widetilde{\CM_I}=g^*\CM-\nu_I E_I,\]
\[f^*F = f_*^{-1} F + E \quad\text{and}\quad f_I^*F_I = (f_I)_*^{-1} F_I + E_I ,\]
where $E$ and $E_I$ are the exceptional divisors of $f$ and $f_I$, and $\widetilde{\CM}$ and $\widetilde{\CM_I}$ are strict transforms of $\CM$ and $\CM_I$ with multiplicities $\nu$, $\nu_I$ positive integers.

Let $g\colon V \to \widetilde{X}$ and $h\colon V \to \widetilde{X_I}$ be the resolution of the flopping map $\psi$.
Denote
\[
F_V = (f\circ g)^*(F) = (f_I\circ h)^*(F_I), \quad \text{and} \quad \CM_V = (f\circ g)_*^{-1}(\CM) = (f_I\circ h)_*^{-1}(\CM_I).
\]
Then we have the following equivalence
\begin{align*}
l F_V + (\frac{1}{2} - \frac{\nu}{n}) g_*^{-1}(E) + \sum a_i E_i \sim K_V + \frac{1}{n} \CM_V\sim l_I F_V + (\frac{1}{2} - \frac{\nu_I}{n}) h_*^{-1}(E_I) + \sum b_i E_i
\end{align*}
for some rational $a_i$ and $b_i$.
Since $\psi$ is small, the divisors $E_i$ in the sums on the left-hand side and on the right-hand side are the same.
We rearrange the terms:
\begin{equation} \label{can-eq-1}
(l-l_I)F_V \sim (\frac{1}{2}-\frac{\nu_I}{n})h_*^{-1}(E_I)+(\frac{\nu}{n}-\frac{1}{2})g_*^{-1}(E) + \sum (b_i - a_i) E_i.
\end{equation}
On the other hand we have
\begin{equation} \label{can-eq-2}
F_V \sim g_*^{-1}(E)+ h_*^{-1}(E_I) + \sum c_i E_i
\end{equation}
for some $c_i$.

It follows from linear independence of $E_i, g_*^{-1}(E), h_*^{-1}(E_I)$ that equivalences (\ref{can-eq-1}) and (\ref{can-eq-2}) are proportional, in particular we have
\[
\nu + \nu_I = n.
\]
It implies that at most one of $(X, \frac{1}{n} \CM)$ and $(X_I, \frac{1}{n} \CM_I)$ can be non-canonical at those particular singular points.
\end{proof}

\begin{Remark}
Note that $\CM_I \subset \big|-n K_{X_I} + l_I F_I \big|$ for some $l_I$, that is $n$ does not change.
This is the case since the maps $\varphi_I$ induce isomorphisms on general fibres.
\end{Remark}

\begin{Lemma}
All the varieties $X_I$ satisfy the generality conditions in the Main Theorem.
\end{Lemma}
\begin{proof}
The maps $\varphi_I$ are compositions of maps $\varphi_i$, where the map $\varphi_i$ is a birational involution of the analytical neighbourhood of the fibre containing $Q_i$.
Thus the fibres, containing singularities on $X_I$ are isomorphic to the fibres containing singularities on $X$. Note that quasi-smoothness of $X_I$ can be checked explicitly. The ambient space is singular along the surface $\{x=y=z=0\}$, which is a $\mathbb{P}(2,2)$-bundle over $\mathbb{P}^1$. This is cut out by $X$ at the solutions to $sw=(\prod_{i \in I} l_i)s - (\prod_{j \not\in I} l_j) w=0\}$, which is a finite set of points. Locally near each singular point $s$ or $w$, as well as a local parameter in the base, appears as linear term (because $l_i$'s are distinct), hence can be eliminated so that the analytic type of singularity at each point is $\frac{1}{2}(1,1,1)$. There are no other singularities by the generality assumptions on $X$.
\end{proof}

\section{Supermaximal singularities}\label{super-max}
In this section we establish a framework for excluding smooth points. 
We follow the approach of Pukhlikov \cite{Pukh123}.
We exclude the points that lie on on $\pi$-vertical curves intersecting $-K_X$ by $\frac{1}{2}$ in sections \ref{staircase}, \ref{multstaircase}, \ref{SMupstairs}, and all the other points in Proposition \ref{PrCurveCenter}.

By Corollary \ref{No-Curves} we may assume that the pair $(X, \frac{1}{n} \mathcal{M})$ is canonical at curves in the smooth locus.
In Section \ref{SecRig} and Section \ref{Involution} we have shown that there is a model $X_I$ for which the pair $(X_I, \frac{1}{n} \mathcal{M}_I)$ is canonical at all the singularities of $X_I$.
By Corollary \ref{CorHPCurve} the pair $(X_I, \frac{1}{n} \mathcal{M}_I)$ is also canonical at curves passing through the singular locus of $X_I$.
Thus if the pair $(X_I, \frac{1}{n} \mathcal{M}_I)$ is not canonical at $B$, then $B$ is a point and $X$ is smooth at $B$. 
This allows us to use the following result.

\begin{Prop}[{\cite[Proposition~2.7]{Okada18}}] \label{supermaximal}
Let $\pi: X \to \MP^1$ be a del Pezzo fibration.
Suppose that we are given a non-square birational map $\chi\colon X \dasharrow Y$ to a Mori fibre space $\pi_Y \colon Y \to Z$ and let $\CM \subset \big| -nK_X + l n F\big|$ be a mobile linear system associated to $\chi$.
Suppose in addition that the pair $(X, \frac{1}{n} \CM)$ is canonical at curves on $X$ and $l \gem 0$. 
Then there exist divisorial valuations $\nu_i$ centred at points $P_1,\dots,P_k$ of $X$ contained in distinct $\pi$-fibres and there exist positive rational numbers $\gamma_1,\dots,\gamma_k$ with the following properties:
\begin{itemize}
\item
$(X,\frac{1}{n} \CM - \sum \gamma_j F_j)$ is not canonical at $\nu_1,\dots,\nu_k$, where $F_i$ is the $\pi$-fibre containing $P_i$.
\item
$\sum_{j=1}^{k} \gamma_j > l$.
\end{itemize}
\end{Prop}

We combine this proposition with Corti's inequality to derive a contradiction from the existence of a non-square birational map $\chi\colon X \dasharrow Y$.
Note that $l \gem 0$ since $X_I$ satisfies the $K^2$-condition.
This is why we require the total $K^2$-condition in the Main Theorem.

We now introduce the notions and notations necessary for application of Corti's inequality.
First, we simplify the notations by dropping the index $I$, that is from now on we denote $X_I$ by $X$, $\pi_I$ by $\pi$, etc.
Let $D_1$, $D_2\in\mathcal{M}$ be general members and consider the effective $1$-cycle $Z=D_1\cdot D_2$.
For some point $P$ we find an upper bound on $\mult_P Z$ using the degrees and a lower bound using Corti's inequality. 

We decompose $Z=Z^v+Z^h$ into the vertical and the horizontal components with relation to the fibration $\pi$. 
For the vertical components we have the decomposition
\begin{align*}
Z^v=\sum_{t\in \MP^1}Z^v_i + Z^{\prime},
\end{align*}
where the support of $Z^v_i$ is contained in $F_i$ from Proposition \ref{supermaximal} and the support of $Z^\prime$ is disjoint from $F_1,\dots,F_k$. 
We define the degree of a vertical $1$-cycle $C^v$ by the number $\deg C^v=C^v\cdot(-K_X)$ and the degree of a horizontal $1$-cycle $C^h$ by the number $\deg C^h=C^h\cdot F$. 
Let $f$ be the class of a line in a fibre, then $\deg f=1$.

Now we look for the pair, and the point, at which we apply Corti's inequality. 

\begin{Lemma}[{\cite[Corollary~4.8]{Kr16}}] \label{CorSMSing}
There is an index $i$ such that 
\begin{align*}
\deg Z^v_i < 4n^2\gamma_i.
\end{align*}
\end{Lemma}

We say that $\nu_i$ is a \emph{supermaximal singularity} if $\deg Z_i^v<4n^2\gamma_i$, and clearly
Lemma \ref{CorSMSing} implies that a supermaximal singularity exists.
Fix such a supermaximal singularity $\nu_i$. To simplify the notations, from now on denote $F_i$ as $F$, $\gamma_i$ as $\gamma$, $P_i$ as $P$, $Z^v_i$ as $Z_v$, $Z^h$ as $Z_h$, and $\nu_i$ as $\nu$.

\begin{Lemma}[{\cite[Proposition~1.1]{Pukh123}}]
Let $P$ be a nonsingular point and let $C$ be a vertical curve on $X$. 
Suppose $\deg C \gem 1$, then $\mult_P C \lem \deg C$.
\end{Lemma}
\begin{proof}
For a sufficiently big integer $a$ the base locus of the linear system $\big|-K_X + aF \big|$ consists of singular points of $X$.
Let $\CH \subset \big|-K_X + aF \big|$ be the subsystem of divisors passing through $P$.
Then $\CH$ has a base curve $L$ if and only if $L$ is the vertical curve of degree $\frac{1}{2}$ passing through $P$. In particular if $\deg C \gem 1$, then $C$ is not in the base locus of $\CH$.
Thus for general $H \in \CH$ we have
$$
\mult_P C \lem H \cdot C = -K_X \cdot C = \deg C.
$$
\end{proof}

\begin{Prop}[{\cite[Section~4]{Pukh123}}] \label{PrCurveCenter}
Let $P$ be a nonsingular point on $X$.
Suppose the degree of every vertical curve passing through $P$ is at least $1$. 
Then $P$ cannot be the centre of a supermaximal singularity.
\end{Prop}
\begin{proof}
By definition of the degree of a horizontal cycle we have
$$
\mult_P Z_h \lem Z_h \cdot F = \deg Z_h = D^2 \cdot F = 2 n^2
$$ 
for a general divisor $D \in \CM$.

Suppose $P$ is the centre of a supermaximal singularity and denote by $F$ the fibre containing $P$. 
A general member $H \in \CH \big\vert_F$ does not share components with $Z_v$ since $\CH\vert_F$ does not have fixed components.
Thus we see that
$$
\mult_P Z_v \lem D \cdot Z_v = \deg Z_v < 4 n^2 \gamma.
$$
On the other hand the pair $(X,\frac{1}{n}\mathcal{M}-\gamma F)$ is not canonical at the point $P$ by Proposition \ref{supermaximal}.
Hence by Corti's inequality there is a number $0 < t \leqslant 1$ such that 
\begin{align*}
\mult_P Z_h + t \mult_P Z_v\geqslant 4n^2(1+\gamma t \mult_P F) > 4(1+\gamma t)n^2.
\end{align*}
Combining the bounds we get a contradiction.
\end{proof}

\begin{Remark} \label{WhyStaircase}
Let $F$ be a fibre not containing a singular point of $X$. Then every curve in $F$ has a positive integer degree and Proposition \ref{PrCurveCenter} applies for every $P \in F$.
Let $F$ be a fibre containing a $\frac{1}{2}(1,1,1)$-point.
Then there are curves on $F$ which intersect $-K_X$ by $\frac{1}{2}$.
If $F$ satisfies the generality condition in the Main Theorem then there are exactly $8$ such curves.
If $P$ lies on such curve, then the bound on multiplicity of $Z_v$ at $P$ becomes $\mult_P Z_v < 8 n^2 \gamma$ which is not enough to lead to a contradiction.
Sections \ref{staircase}, \ref{multstaircase}, and \ref{SMupstairs} are entirely devoted to dealing with this issue.
\end{Remark}

\section{Construction of the staircase in the generic case} \label{staircase}

Let $X^{(0)}$ be a projective threefold, let $F^{(0)}\subset X^{(0)}$ be a surface, let $L_0\subset F^{(0)}$, $L_0 \cong \MP^1$, and suppose $X^{(0)}$ and $F^{(0)}$ are smooth in the neighbourhood of $L_0$.
Let us associate the following construction to $L_0$ which goes by the name \emph{staircase}.
The notion of the staircase has been introduced by Pukhlikov in \cite[Section~6]{Pukh123}. 
The staircase associated to $L_0$ is the following chain of morphisms which we define inductively
\begin{displaymath}
\xymatrix
{ 
	X^{(M)} ~ \ar[r]^{\sigma_M} & ~ X^{(M-1)} ~ \ar[r]^{\sigma_{M-1}} & ~ \dots ~ \ar[r]^{\sigma_1} &  ~ X^{(0)}.
}
\end{displaymath}

The morphism $\sigma_i:X^{(i)}\to X^{(i-1)}$ is the blow up at $L_{i-1}$ and we denote $E^{(i)}$ to be its exceptional divisor.
Clearly for every $i$ we have that $E^{(i)}\cong\mathds{F}_m$, for some $m$. 
If $m>0$, then let $L_i$ be the exceptional section of $E^{(i)}$.
If $m=0$, then let $L_i$ be a line from the ruling of $E^{(i)} \cong \MP^1 \times \MP^1$ horizontal with respect to the $\MP^1$-fibration $\sigma_i\vert_{E^{(i)}}$.
Denote the proper transform of $E^{(i)}$ on $X^{(j)}$, as $E^{(i,j)}$ and the proper transform of $F^{(0)}$ on $X^{(j)}$ as $F^{(j)}$.
Let $f_i \in A^2(X^{(i)})$ be the class of a fibre of the ruled surface $E^{(i)}$. 

\begin{Th} \label{ThStaircase}
Let $X^{(0)}$ be a  projective threefold and let $F^{(0)}$ be a surface in it and suppose $L_0$ is a smooth rational curve in $F^{(0)}$. Suppose $X^{(0)}$ and $F^{(0)}$ are smooth in the neighborhood of $L_0$. Also assume that $L_0\cdot K_{X^{(0)}}=0$ and $L_0 \cdot F^{(0)}=-1$. Then for the staircase associated to $L_0$ the following assertions are true:
\begin{enumerate}[(i)]
	\item $E^{(1)} \cong \MP^1 \times \MP^1$,
	\item $E^{(i)} \cong \mathds{F}_1$ for $i \gem 2$,
	\item $\sigma_1^*(L_{0})\equiv L_1-f_1$,
	\item $\sigma_i^*(L_{i-1}) \equiv L_{i}$ for $i \gem 2$,
	\item $L_i \subset E^{(1,i)}$ for all $i$,
	\item $E^{(i,i+1)}\vert_{E^{(i+1)}}$ is disjoint from $L_{i+1}$ for $i\geqslant 2$,
	\item $F^{(1)}$ either contains $L_1$ or is disjoint from it, and $\nu_{E^{(1)}}(F^{(0)}) = 1$,
	\item $F^{(i)}$ is disjoint from $L_i$ for $i \gem 2$, and $\nu_{E^{(i)}}(F^{(0)}) = i-1+\delta$, where $\delta = \mult_{L_1} F^{(1)}$.
\end{enumerate}
\end{Th}


\begin{center}
\begin{tikzpicture}[scale=0.7]
\draw (-10,0)--(-10,2);
\draw (-10,2)--(-8.5,3);
\node at (-9.5,2.6) {\tiny$F$};
\draw (-8.5,3)--(-8.5,1);
\draw (-8.5,1)--(-10,0);

\draw (-8.5,3)--(-6.5,3);
\node at (-6.8,3.2) {\tiny$E_1$};
\draw (-6.5,3)--(-6.5,1);
\draw(-6.5,1)--(-8.5,1);

\draw (-7.5,3)[red]--(-7.5,1);
\node at (-7.8,2) {\tiny$L_n$};

\draw (-7.5,3)--(-5.8,1.5);
\node at (-6,2.1) {\tiny$E_n$};
\draw (-5.8,1.5)--(-5.8,-0.5);
\draw (-5.8,-0.5)--(-7.5,1);

\draw (-5.8,1.5)--(-3.8,1.5);
\node at (-4.5,1.7) {\tiny$E_{n-1}$};
\draw (-3.8,1.5)--(-3.8,-0.5);
\draw (-3.8,-0.5)--(-5.8,-0.5);

\draw (-3.8,1.5)--(-3.6,1.7);
\draw (-3.8,-0.5)--(-3.6,-0.3);

\draw[dashed] (-3.6,1.7)--(-2.9,2.4);
\draw[dashed] (-3.6,-0.3)--(-2.9,0.4);

\draw (-2.8,2.5)--(-2.6,2.7);
\draw (-2.8,0.5)--(-2.6,0.7);
\draw (-2.6,2.7)--(-2.6,0.7);

\draw (-2.6,2.7)--(-0.6,2.7);
\node at (-0.8,2.9) {\tiny$E_2$};
\draw (-0.6,2.7)--(-0.6,0.7);
\draw (-0.6,0.7)--(-2.6,0.7);

\node  [align=flush center,text width=5.5cm] at (-4.7,-1) {\footnotesize The staircase when $\delta = 0$};

\draw (2,0)--(2,2);
\draw (2,2)--(3.5,3);
\node at (2.5,2.6) {\tiny$F$};
\draw (3.5,3)--(3.5,1);
\draw (3.5,1)--(2,0);

\draw (3.5,3)--(5,2);
\node at (4.7,2.5) {\tiny$E_2$};
\draw (5,2)--(5,0);
\draw (5,0)--(3.5,1);

\draw (5,2)--(6.5,3);
\node at (5.8,2.9) {\tiny$E_3$};
\draw (6.5,3)--(6.5,1);
\draw (6.5,1)--(5,0);

\draw (6.5,3)--(6.9,3);
\draw[dashed] (7,3)--(8,3);
\draw (8.1,3)--(8.5,3);
\draw (8.5,3)--(8.5,1);
\draw (8.5,1)--(8.1,1);
\draw[dashed] (7,1)--(8,1);
\draw (6.9,1)--(6.5,1);

\draw (8.5,3)--(10,2);
\node at (9.7,2.5) {\tiny$E_n$};
\draw[red] (10,2)--(10,0);
\node at(10.3,1.2) {\tiny$L_n$};
\draw (10,0)--(8.5,1);

\draw (10,2)--(11.5,3);
\node at (10.8,2.9) {\tiny$E_1$};
\draw (11.5,3)--(11.5,1);
\draw (11.5,1)--(10,0);

\node  [align=flush center,text width=5.5cm] at (7.3,-1) {\footnotesize The staircase when $\delta = 1$};

\end{tikzpicture}
\end{center}

We use the following result in the proof of the theorem.

\begin{Lemma}[{\cite[Lemma~2.2.14]{IPFano}}] \label{LemExcInt}
Let $\sigma:Y\to X$ be the blow up of a threefold $X$ at a smooth projective curve $C$.
Suppose also that $X$ is smooth along $C$.
Let $E$ be the exceptional divisor of $\sigma$, then $E$ is a projectivization of the normal bundle $N_{C/X}$. 
Denote the class of the fibre of $E$ as $f$, then the following equalities hold
\begin{enumerate}[(i)]
	\item $E^2=-\sigma^*(C)+\deg(N_{C/X})f$,
	\item $E^3=\deg(N_{C/X})$,
	\item $E\cdot f=-1$,
	\item $E\cdot \sigma^*(D)=(C\cdot D)f$, for some divisor $D$ on $X$,
	\item $f\cdot \sigma^*(D)=0$, for some divisor $D$ on $X$,
	\item $E\cdot \sigma^*(Z)=0$, for any $Z \in A^2(X)$,
	\item $\deg(N_{C/X})=2g(C)-2-K_X\cdot C$.
\end{enumerate}
\end{Lemma}

To prove Theorem \ref{ThStaircase} we combine the following lemmas.

\begin{Lemma} \label{LemStaircase1}
With the assumptions of Theorem \ref{ThStaircase}, we have
$\deg N_{L_0/X^{(0)}}=\mathcal{O}_{L_0}(-1)\oplus\mathcal{O}_{L_0}(-1)$.
In particular, $E^{(1)} \cong \MP^1 \times \MP^1$.
\end{Lemma}
\begin{proof}
Recall that $L_0$ is a smooth rational curve, thus Lemma \ref{LemExcInt} implies:
\begin{align*}
\deg N_{L_0/X^{(0)}} = 2g(L_0) - 2 - K_{X^{(0)}} \cdot L_0 = -2.
\end{align*}
Since $F^{(0)}\cdot L_0 = -1$ we have
$(N_{F^{(0)}/X^{(0)}})\big\vert_{L_0} \cong \mathcal{O}_{L_0}(-1)$.
There is the exact sequence
\begin{align*}
0\to N_{L_0/F^{(0)}} \overset{e}{\to} N_{L_0/X^{(0)}} \overset{s}{\to} (N_{F^{(0)}/X^{(0)}})\big\vert_{L_0} \to 0.
\end{align*}
For some $a$ and $b$ we have $N_{L_0/X^{(0)}} \cong \mathcal{O}_{L_0}(a)\oplus\mathcal{O}_{L_0}(b)$, without loss of generality we may assume that $a\lem  b$. 
Because $\deg N_{L_0/X^{(0)}} = -2$ we have $a + b = -2$ and since $(N_{F^{(0)}/X^{(0)}})\big\vert_{L_0} \cong \mathcal{O}_{L_0}(-1)$
we have $N_{L_0/F^{(0)}} \cong \mathcal{O}_{L_0}(-1)$.

Let $p_a$ be the projection $\mathcal{O}_{L_0}(a)\oplus\mathcal{O}_{L_0}(b) \to \mathcal{O}_{L_0}(a)$.
Suppose $p_a \circ e \colon \mathcal{O}_{L_0}(-1) \to \mathcal{O}_{L_0}(a)$ is a nontrivial map, then $b \gem a \gem -1$ and therefore $a = b = -1$.
Hence $a = -1$ and $b = -1$.

Suppose $p_a \circ e$ is trivial, then the map $s \vert_{\mathcal{O}_{L_0}(b)}$ is nontrivial and $b \lem -1$. Thus we have $b = -1$ and hence $a = b = -1$.
\end{proof}

\begin{Lemma} \label{LemStaircase2}
With the assumptions of the Theorem \ref{ThStaircase}, we have 
\begin{enumerate}[(i)]
	\item $\sigma_1^*(L_{0}) \equiv L_1 - f_1$,
	\item $E^{(1)}\cdot L_1=-1$,
	\item $K_{X^{(1)}} \cdot L_1 = -1$, and
	\item $E^{(2)} \cong \mathds{F}_1$.
\end{enumerate}
\end{Lemma}
\begin{proof}
By Lemma \ref{LemExcInt} we have
\begin{align*}
0 = E^{(1)} \cdot \sigma_1^*(L_{0}) = E^{(1)}\vert_{E^{(1)}} \cdot \sigma_1^*(L_{0}) = 
\big( -\sigma_1^*(L_{0})-2f_1 \big) \cdot \sigma_1^*(L_{0}),
\end{align*}
thus $\sigma_1^*(L_{0})^2 = -2f_1 \cdot \sigma_1^*(L_{0})$.
Clearly $\sigma_1^*(L_{0}) \equiv L_1 + a f_1$ for some $a$, hence
\begin{align*}
-2 = \sigma_1^*(L_0)^2 = (L_1 + af_1)^2 = 2a.
\end{align*}
Therefore $\sigma_1^*(L_0) \equiv L_1 - f_1$.

We use $(i)$ and Lemma \ref{LemExcInt} to prove $(ii)$:
\begin{align*}
L_1 \cdot E^{(1)} = \sigma_1^*(L_0) \cdot E^{(1)} + f_1 \cdot E^{(1)} = 
f_1 \cdot E^{(1)} = -1.
\end{align*}

Thus $(iii)$ holds:
\begin{align*}
L_1 \cdot K_{X^{(1)}} = L_0 \cdot K_{X^{(0)}} + f_1 \cdot E^{(1)} = -1.
\end{align*}

By Lemma \ref{LemExcInt} we compute
\begin{align*}
\deg N_{L_1/X^{(1)}} = -2 - K_{X^{(1)}} \cdot L_1 = -1.
\end{align*}
On the other hand $N_{ E^{(1)} / X^{(1)} }\vert_{L_1} \cong \mathcal{O}_{L_1}(-1)$ by $(ii)$.
Using the same argument as in Lemma \ref{LemStaircase1} we conclude that $E^{(2)} \cong  \mathds{F}_1$.
\end{proof}

\begin{Lemma} \label{LemStaircase3}
\begin{enumerate}[(i)]
Suppose $E^{(i)} \cong \mathds{F}_1$ and $K_{X^{(i-1)}} \cdot L_{i-1} = -1$, then
	\item $\sigma_i^*(L_{i-1}) \equiv L_{i}$,
	\item $E^{(i)} \cdot L_{i} = 0$,
	\item $K_{X^{(i)}} \cdot L_{i} = -1$, and
	\item $E^{(i+1)} \cong \mathds{F}_1$.
\end{enumerate}
\end{Lemma}
\begin{proof}
Since $\deg N_{L_{i-1}/X^{(i-1)}}=-1$, Lemma \ref{LemExcInt} implies
\begin{align*}
0 = E^{(i)} \cdot \sigma_i^*(L_{i-1}) = 
E^{(i)} \vert_{E^{(i)}} \cdot \sigma_i^*(L_{i-1}) = 
\big( -\sigma_i^*(L_{i-1}) - f_i \big) \cdot \sigma_i^*(L_{i-1}).
\end{align*}
Therefore $\sigma_i^*(L_{i-1})^2=-f\cdot\sigma_i^*(L_{i-1})=-1$, and hence $\sigma_i^* (L_{i-1})$ is equivalent to the exceptional section $L_i$. 
Thus assertions $(i)$ and $(ii)$ hold.

By $(i)$ we have
\begin{align*}
K_{X^{(i)}}\cdot L_{i} = K_{X^{(i)}}\cdot \sigma_{i}^*(L_{i-1})=K_{X^{(i-1)}}\cdot L_{i-1}+E^{(i)}\cdot \sigma_i^*(L_{i-1})=-1.
\end{align*}

Lemma \ref{LemExcInt} implies 
$$
\deg N_{L_i/X^{(i)}} = -2 - K_{X^{(i)}} \cdot L_i = -1,
$$ 
while $N_{L_i/E^{(i)}} \cong \mathcal{O}_{L_i}(-1)$ since $L_i^2 = -1$. 
Hence computing $N_{L_{i}/X^{(i)}}$ as in Lemma \ref{LemStaircase1} we conclude that $E^{(i+1)} \cong \mathds{F}_1$.
\end{proof}

\begin{proof}[Proof of Theorem \ref{ThStaircase}:]
Lemma \ref{LemStaircase1} implies $(i)$.
Lemma \ref{LemStaircase2} implies $(iii)$.
Lemma \ref{LemStaircase2} and Lemma \ref{LemStaircase3} imply $(ii)$ and $(iv)$ by induction.

We prove $(v)$ by induction.
Obviously $L_1 \subset E^{(1)}$.
The equality $L_1\cdot E^{(1)} = -1$ holds by Lemma \ref{LemStaircase2}.
We show that if $L_{i-1}\cdot E^{(1,i-1)}=-1$, then $L_i\subset E^{(1,i)}$ and $L_i\cdot E^{(1,i)}=-1$. 
Indeed by Lemma \ref{LemExcInt} we have
\begin{align*}
E^{(1,i)}\cdot E^{(i)} \equiv \sigma_i^*(E^{(1,i-1)}) \cdot E^{(i)} - (E^{(i)})^2 \equiv
\big( E^{(1,i-1)} \cdot L_{i-1} \big) f_i + L_i + f_i \equiv L_i.
\end{align*}
The curve $L_i$ is the only effective curve in its numerical equivalence class that lies on $E^{(i)}$ since $E^{(i)} \cong \MF_1$ by assertion $(ii)$.
Thus we see that $E^{(i)} \cap E^{(1,i)} = L_i$, in particular $L_i\subset E^{(1,i)}$. 
Also 
\begin{align*}
E^{(1,i)} \cdot L_i = \big( E^{(1,i)}\vert_{E^{(i)}} \cdot L_i \big)_{E^{(i)}} =
L_i^2 = -1.
\end{align*}
Thus $(v)$ holds.

By induction Lemma \ref{LemStaircase2} and Lemma \ref{LemStaircase3} imply that $E^{(i)} \cdot L_{i} = 0$ for all $i \gem 2$. Thus we compute
\begin{align*}
E^{(i-1,i)} \cdot E^{(i)} \equiv \sigma_i^*(E^{(i-1)}) \cdot E^{(i)} - (E^{(i)})^2 \equiv \big( E^{(i-1)} \cdot L_{i-1} \big) f_i + L_i + f_i \equiv L_i + f_i.
\end{align*}
Hence by computing the intersection
\begin{align*}
E^{(i-1,i)} \cdot L_i = E^{(i-1,i)}\vert_{E^{(i)}} \cdot L_i 
= \big( (L_i + f_i)  \cdot L_i \big)_{E^{(i)}}=0.
\end{align*}
We conclude that either $E^{(i-1,i)}$ is disjoint from $L_i$ or $L_i \subset E^{(i-1,i)}$.
The latter does not occur since $E^{(1,i-1)}$ intersects $E^{(i-1)}$ transversally and $L_i \subset E^{(1,i)}$.
Therefore we have shown that $(vi)$ holds.

By Lemma \ref{LemExcInt} we have
\begin{align*}
F^{(1)}\vert_{E^{(1)}} \equiv 
- \big( E^{(1)} \big)^2 + \big( F^{(0)} \cdot L_0 \big) f_1 
\equiv \sigma_1^*(L_0) + 2 f_1 - f_1 \equiv L_1.
\end{align*}
Using this equivalence we compute $F^{(1)} \cdot L_1 = \big( L_1 \cdot L_1 \big)_{E^{(1)}} = 0$. 
Thus $F^{(1)}$ is either disjoint from $L_1$ or it contains $L_1$.
Since $\mult_{L_0} F^{(0)} =1$ we have $\nu_{E^{(1)}} (F^{(0)}) = 1$ that is to say $(vii)$ holds.

If $F^{(1)}$ and $L_1$ are disjoint, then $F^{(i)}$ is disjoint from $L_i$. 
Suppose $L_1 \subset F^{(1)}$, then $F^{(1)}$ and $E^{(1)}$ are smooth surfaces intersecting along a smooth curve $L_1$ transversally at every point.
Thus $F^{(2)}$ and $E^{(1,2)}$ are disjoint, in particular $L_2 \subset E^{(1,2)}$ and $F^{(2)}$ are disjoint.
It follows that $L_i$ is disjoint from $F^{(i)}$ for all $i\gem 2$.

To finish the proof of $(viii)$ we compute
$$
\big( \sigma_{1} \circ \sigma_{2} \circ \dots \circ \sigma_i \big)^* (F^{(0)}) = F^{(i)} + E^{(1,i)} + (1 + \delta) E^{(2,i)} + \dots + (i - 2 + \delta) E^{(i-1,i)} + (i - 1 + \delta) E^{(i)}.
$$
\end{proof}

\subsection*{The staircase over a del Pezzo fibration}
We would like to apply Theorem \ref{ThStaircase} to $\pi_I \colon X_I \to \MP^1$.
As before, we drop the index $I$ to simplify the notations.
Suppose $\pi \colon X \to \MP^1$ is a del Pezzo fibration of degree $2$ such that $X$ is a quasi-smooth complete intersection in a $\MP(1,1,1,2,2)$-bundle over $\MP^1$. 
Suppose $Q\in X$ is a $\frac{1}{2}(1,1,1)$-point and $F$ is the fibre containing $Q$. 
Let $\sigma_Q: X^{(0)} \to X$ be the Kawamata blow up of $X$ at $Q$ and let 
$E_Q$ be the exceptional divisor of $\sigma_Q$.
The fibre $F$, containing $Q$, can be embedded into $\MP(1,1,1,2)$ and the equation of the image is of the form 
$$
w q(x,y,z) + r(x,y,z)= 0.
$$

Let $L\subset F$ be a ``half-line'', that is a curve $L$ such that 
$L \cdot (- K_F) = \frac{1}{2}$. 
Let $L_0$ and $F^{(0)}$ be the proper transforms of $L$ and $F$ respectively on $X^{(0)}$.
We may construct the staircase associated to $L_0$. 
We also say that the staircase is associated to the half-line $L$. 

If $X$ satisfies the generality conditions, then $F^{(0)}$ and $L^{(0)}$ are smooth along $L_0$.
Thus to show that the triple $X^{(0)}$, $F^{(0)}$, $L_0$ satisfies the assumptions of Theorem \ref{ThStaircase} we only need to consider the intersections.

\begin{Lemma} \label{LemX0}
The following equalities hold
\begin{enumerate}[(i)]
	\item $L_0\cdot E_Q = 1$,
	\item $L_0\cdot F^{(0)} = -1$,
	\item $L_0 \cdot K_{X^{(0)}} = 0$.
\end{enumerate}
\end{Lemma}
\begin{proof}
Let $\MP = \MP(1,1,1,2)$, and let $H$ be the generator of $\Cl(\MP)$. 
Because $F \subset \MP$ we can express $L$ in $\MP$ as the intersection $H_1\cdot H_2$, for some $H_i\in\big| H \big|$. 
Let $\sigma_\MP: \MP_Q \to \MP$ be the blow up at the point $Q$ and let $E_{\MP}$ be its exceptional divisor. 
Clearly $\sigma_\MP:(\sigma_\MP)^{-1}_*(F)\to F$ is the blow up at $Q$ thus we may identify $(\sigma_\MP)^{-1}_*(F)$ with $F^{(0)}$ and $(\sigma_\MP)^{-1}_*(L)$ with $L_0$. 
Let $H_{i,Q}$ be the proper transform of $H_i$ on $\MP_Q$, then 
$L_0 = H_{1,Q} \cdot H_{2,Q}$.
Denote the exceptional divisor of $\sigma_\MP$ as $E_\MP$, then
\begin{align*}
L_0\cdot E_Q = L_0\cdot E_Q \big\vert_{F^{(0)}} = L_0 \cdot E_\MP \big\vert_{F^{(0)}} 
= L_0 \cdot E_\MP = H_{1,Q} \cdot H_{2,Q} \cdot E_\MP =
H_{1,Q} \big\vert_{H_{2,Q}} \cdot E_\MP \big\vert_{H_{2,Q}}.
\end{align*}
The surface $H_2$ is isomorphic to $\MP(1,1,2)$ and $\sigma_\MP\vert_{H_{2,Q}}$ is the blow up at the singular point. 
It follows that $H_{2,Q} \cong \mathds{F}_2$, $E_\MP \big\vert_{H_{2,Q}}$ is the exceptional section, and $H_{1,Q} \big\vert_{H_{2,Q}}$ is a fibre. 
Thus we have $(i)$
\begin{align*}
L_0\cdot E_Q
= H_{1,Q} \big\vert_{H_{2,Q}} \cdot E_\MP \big\vert_{H_{2,Q}} = 1.
\end{align*}

To prove $(iii)$ it is enough to show that $F^{(0)}=F-\frac{2}{2}E_Q=F-E_Q$.
Indeed we find the intersection
\begin{align*}
L_0 \cdot F^{(0)} = L_0 \cdot (F - E_Q) = -L_0\cdot E_Q = -1.
\end{align*}
The equality $(iii)$ follows from $(i)$
\begin{align*}
K_{X^{(0)}} \cdot L_0 = (\sigma^*{K_X} + \frac{1}{2} E_Q) \cdot L_0 
= K_X \cdot L + \frac{1}{2} E_Q \cdot L_0 = 0.
\end{align*}

We now prove the claim.
Recall that $X$ is a complete intersection in $\MP(1,1,1,2,2)$ given by the equations
\begin{align*}
g(u,v) s = f(u,v) w + q(x,y,z;u,v),\quad \text{and} \quad
s w + r(x,y,z;u,v) = 0.
\end{align*}
The point $Q$ lies on $(x=y=z=0)$, up to a change of coordinates on the base, we may assume that it is in the fibre $u=0$.
Thus up to exchanging variables $w$ and $s$ we may assume that $Q$ is given by $x=y=z=w=u$.
Consider a neighbourhood $U_V$ of $Q$ given by $s=v=1$, then the equations of $U = X \cap U_V$ are
\begin{align*}
&u(1+u(\dots)) = w (1+u(\dots)) + q(x,y,z;u,1),\quad \text{and} \quad
&w + r(x,y,z;u,1) = 0.
\end{align*}
We may now eliminate the second equation to acquire the embedding $U \subset T_U \cong \MC^3_{x,y,z}/\mu_2 \times \MC_u$ given by the equation
$$
u(1+u(\dots)) + r(x,y,z;u,1) (1+u(\dots)) = q(x,y,z;u,1).
$$
We consider an analytic neigbourhood of $0$ in $\MC^3_{x,y,z}/\mu_2$ and the corresponding analytic neighbourhoods $U_{an}$ and $T_{U,an}$.
The equation of $U_{an}$ is $u=0$ and hence the equation of $F_{an} = F \cap U_{an}$ is
$$
r(x,y,z;0,1) =  q(x,y,z;0,1),
$$
where $r(x,y,z;0,1) \neq 0$ by the generality conditions.
It follows that $F^{(0)}=F-\frac{2}{2}E_Q=F-E_Q$ since $\deg r(x,y,z;0,1) = 2$.
\end{proof}

We now discuss the termination condition of the staircase.
Recall that there is a mobile linear system $\mathcal{M}\subset\big| -n K_X + l F\big|$ and discrete valuations $\nu$ such that the pair $(X,\frac{1}{n}\mathcal{M})$ is not canonical at a discrete valuation $\nu$, and $\nu$ is a supermaximal singularity. 
We have shown in Section \ref{SecRig} and Section \ref{Involution} that the centre of $\nu$ is a nonsingular point $P$. 
By Remark \ref{WhyStaircase} and Proposition \ref{PrCurveCenter} the point $P$ lies on a curve $L$ such that $L \cdot K_{X_I} = - \frac{1}{2}$.
In particular $P$ is in a fibre containing a $\frac{1}{2}(1,1,1)$-point $Q$.
We want to track the centre of $\nu$ as we go up the staircase, in particular we want to know how far up the staircase we must go.

\begin{Prop}[{\cite[Proposition~7.1]{Pukh123}}] \label{PropStopCond}
Let $X^{(0)}$ be a threefold and let $F^{(0)}$ be a surface in it. Suppose $L_0 \cong \MP^1$, $L_0 \subset F^{(0)}$. 
Suppose also $X^{(0)}$ and $F^{(0)}$ are smooth in the neighbourhood of $L_0$.
Let $\sigma_i: X^{(i)} \to X^{(i-1)}$ be the associated staircase. Let $\nu$ be a discrete valuation of $K(X^{(0)})$ and suppose that a centre of $\nu$ on $X^{(0)}$ is a point on $L_0$. 
Then there is a positive integer $M$ such that for every $i<M$ the centre of $\nu$ on $X^{(i)}$ is a point on $L_i$ and the centre of $\nu$ on $X^{(M)}$ is one of the following
\begin{enumerate}[A)]
	\item a fibre of a ruled surface $E^{(M)}$,
	\item a point not on $L_M$ and not on $E^{(M-1,M)}$ and $M \gem 2$, or
	\item a point on $E^{(M)}\cap E^{(M-1,M)}$ and $M \gem 3$.
\end{enumerate}
\end{Prop}

The idea of the proof is the following: the discrepancy of $E^{(i)}$ is increasing, therefore the centre of $\nu$ should be away from $E^{(i)}$ for $i \gg 0$.
If the centre of $\nu$ on $E^{(1)}$ is a point we can choose $L_1$ to be the line from horizontal ruling passing through that point, thus we never get cases B and C for $M = 1$.
The intersection of $E^{(1,2)}$ and $E^{(2)}$ is $L_2$, thus we never get the case $C$ for $M=2$.

\section{Multiplicities on the staircase} \label{multstaircase}
The plan is to associate a staircase to a half-line, to apply Corti's inequality upstairs, and to derive a contradiction. 
Thus we need to find bounds on multiplicities of the cycles upstairs.
We compute them in this section.

Let $X$ be a projective threefold, let $B$ be a smooth, irreducible subvariety of codimension $2$, and suppose $X$ is smooth along $B$.
Let $\sigma: \widetilde{X} \to X$ be the blow up along $B$.
Then by definition of the proper transform for any cycle $Z$ such that $B \not\subset \ON{Supp} Z$ we have
$\sigma^*(Z) = \sigma^{-1}_*(Z) + Z_E$,
where $Z_E$ is the cycle with the support on the exceptional divisor $E$ of $\sigma$.
The following well known results give us some information about $Z_E$.

\begin{Lemma} \label{LemCycle1}
Suppose $B$ is a smooth curve then $Z_E\equiv(Z\cdot B)_S f$, where $f$ is the class of a fibre of the ruled surface $E$ and $S$ is any normal surface containing $\ON{Supp} Z$ and $B$ which is smooth at every point of $\ON{Supp} Z\cap B$.
\end{Lemma}

\begin{Lemma}[{\cite[Lemma~2.14]{Kr16}}] \label{LemCycle2}
Let $F$ be the hyperplane given by the equation $z=0$ in $\MC^3$, let $B$ be a smooth irreducible curve in $F$, and let $C$ be an irreducible curve which does not lie in $F$. 
Let $f\in A^{2}(X)$ be the class of a fibre of the ruled surface $E$. 
Then $\sigma^* C \equiv \sigma^{-1}_* C+kf$, where $k\leqslant C\cdot F$.
\end{Lemma}

Suppose cycle $Z$ is the intersection $D_1 \cdot D_2$.
Then we can express $\sigma^*Z$ in terms of the intersection of proper transforms of $D_1$ and $D_2$ and a cycle supported on the exceptional divisor.

\begin{Lemma} \label{LemD2}
Let $D_1$ and $D_2$ be general members in a mobile linear system $\mathcal{M}$ on a smooth variety $X$ and suppose $Z=D_1\cdot D_2$. 
Let $\widetilde{D}_i$ be the proper transform of $D_i$ on $\widetilde{X}$. Then
\begin{align*}
\widetilde{D}_1\cdot \widetilde{D}_2\equiv \sigma^{*}(Z)+\Delta_E,
\end{align*}
where $\ON{Supp} \Delta_E\subset E$. 

Suppose also that $B$ is a smooth curve. Let $m=\nu_E(\mathcal{M})$ and let $f\in A^2(\widetilde{X})$ be the class of a fibre of the ruled surface $E$. Then 
\begin{align*}
\Delta_E\equiv m^2E^2-2m(D_1\cdot B)f.
\end{align*}
\end{Lemma}

We use the following notations for the proper transforms on the staircase.
Let $A$ be a cycle, a divisor, or a linear system on $X$. 
We denote its proper transform on $X^{(i)}$ as $A^{(i)}$. 
For divisors and cycles on $X^{(j)}$ we add upper index. For example, $E^{(1,3)}$ is the proper transform of $E^{(1)}$ on $X^{(3)}$. By $\sigma^*$ we mean the appropriate composition of $\sigma^*_i$. 
For example, $E^{(1,3)}=\sigma^*(E^{(1)})-E^{(2,3)} - 2 E^{(3)}$, here $\sigma^* =\sigma^*_3\circ\sigma_2^*$. 

Recall the notations of Section 5. Let $Z=D_1\cdot D_2$ for general members $D_1,D_2\in\mathcal{M}$. Let $Z_h$ be the horizontal part of $Z$ and let $Z_v$ be the part of $Z$ which lies in $F$.
Let $\gamma$ be the number such that the pair 
$(X,\frac{1}{n}\mathcal{M}-\gamma F)$
is not canonical at $\nu$ and $\deg Z_v < 4n^2\gamma$. 
Let $L$ be the curve passing through $P$ and satisfying $L \cdot K_X = - \frac{1}{2}$.
The cycle $Z_v$ can be decomposed as 
$$
Z_v=kL+\Delta,
$$
where $k\geqslant 0$ and $\Delta$ does not contain $L$.

\begin{Lemma} \label{LemCLInt}
The inequality $\Delta \cdot L < 4 \gamma n^2$ holds.
\end{Lemma}
\begin{proof}
Let $\mathcal{H}$ be the subsystem of $\big| -K_F \big|$ of divisors containing $L$.
We claim that $\dim \mathcal{H} = 1$.
Pick a point $P \in L$ such that $X$ is smooth at $P$.
Then it is easy to see that the linear system of divisors in $\big| -K_F \big|$ containing $P$ has dimension $1$.
On the other hand, any divisor in $\big| -K_F \big|$ containing $P$ also contains $L$.
Let $H$ be a general element of $\mathcal{H}$.
Let $L^\prime=H-L$. 
Since $\mathcal{H}$ does not have fixed points except points on $L$, the linear system $\big| L^\prime \big|$ is mobile, in particular $L^\prime$ is numerically effective. 
The valuation $\nu$ is a supermaximal singularity hence we have $\deg Z_v < 4\gamma n^2$ and therefore 
\begin{align*}
\Delta \cdot L \leqslant \Delta \cdot L+ \Delta \cdot L^\prime = \Delta \cdot H \leqslant Z_v \cdot H = \deg Z_v < 4\gamma n^2
\end{align*}
\end{proof}

Denote $\nu_Q = \nu_{E_Q}$.
Since the pair $(X,\frac{1}{n}\mathcal{M})$ is canonical at $Q$ we have $\nu_{Q}(D)\leqslant \frac{n}{2}$ for general $D\in\mathcal{M}$.

\begin{Lemma}
Suppose $D$ is a general member in $\mathcal{M}$, then
\begin{align*}
&D^{(0)}\cdot L_{0}=\frac{n}{2}-\nu_Q(D), \quad and\\
&D^{(i)}\cdot L_{i}=\frac{n}{2}-\nu_Q(D)+\lambda_1 \quad \text{for } i\geqslant 1,
\end{align*}
where $\lambda_1 = \mult_{L_0} \mathcal{M}^{(0)}$.
\end{Lemma}
\begin{proof}
Lemma \ref{LemX0} implies
\begin{align*}
D^{(0)}\cdot L_0&=\sigma_Q^*(D)\cdot L_0 - \nu_Q(D)E_Q\cdot L_0=D\cdot L - \nu_Q(D)=\\
&-nK_X\cdot L - \nu_Q(D)=\frac{n}{2}-\nu_Q(D).
\end{align*}
On the other hand by Lemma \ref{LemExcInt} and $(iii)$ of Theorem \ref{ThStaircase} 
\begin{align*}
D^{(1)}\cdot L_1 = (\sigma_1^* D^{(0)} - \lambda_1 E^{(1)}) (\sigma_1^* L_0 + f_1)
= D^{(0)} \cdot L_0 + \lambda_1.
\end{align*}
By Theorem \ref{ThStaircase} we have $\sigma_i^*L_{i}=L_{i+1}$ for $i\geqslant 1$.
Thus the equality $D^{(1)}\cdot L_{1}=D^{(i)}\cdot L_{i}$ holds for all $i\geqslant 1$. 

\end{proof}

Denote $Z_i=D^{(i)}_1\cdot D^{(i)}_2$, we have the decomposition $Z_0=Z_v^{(0)}+Z_h^{(0)}+Z_Q$, where $Z_Q \subset E_Q$.
We can disregard the part $Z_Q$ because it is away from $P$.

For every $Z_i$ we have the part of the cycle in $E^{(i)}$, let us denote it $\Gamma^{(i)}$. 
Recall that $E^{(1)} \cong \MP^1 \times \MP^1$ and $E^{(i)} \cong \mathds{F}_1$ for $i\geqslant 2$. 
The map $\sigma_i \big\vert_{E^{(i)}}$ is the corresponding $\MP^1$-fibration. 
We say that a curve $B$ on $E^{(i)}$ is vertical if $\sigma_i(B)$ is a point and horizontal otherwise. 
We can decompose the cycle $\Gamma^{(i)}$ into $L_i$ with multiplicity, the rest of the horizontal part, and the vertical part: 
$$
\Gamma^{(i)} = k_iL_i+C^{(i)}_h+C^{(i)}_v. 
$$
Since $E^{(i,i+h)}$ is disjoint from $L_{i+h}$ for any $h\geqslant 2$, $i\geqslant 2$, we have $\sigma^* C^{(i,i+1)}\equiv C^{(i,i+h)}$ for any $h\geqslant 2$, $i\geqslant 2$.
Similarly $\Delta^{(i)}\equiv \sigma^* \Delta^{(2)}$ for $i\geqslant 3$ since $F^{(i)}$ is disjoint from $L_i$ for $i \gem 2$. 
Thus we can decompose 
\begin{align*}
Z_0 &= Z_h^{(0)} + Z_v^{(0)} = Z_h^{(0)} + \Delta^{(0)} + k_0 L_0,\\
Z_1 &= Z_h^{(1)} + \Delta^{(1)} + C_h^{(1)} + C_v^{(1)} + k_1L_1,\\
Z_2 &= Z_h^{(2)} + \Delta^{(2)} + C_h^{(1,2)} + C_v^{(1,2)} + 
C_h^{(2)} + C_v^{(2)} + k_2L_2,\\
Z_i &= Z_h^{(i)} + \sigma^* \big( \Delta^{(2)} \big) + C_h^{(1,i)} + C_v^{(1,i)} 
+ \sigma^*C_h^{(2,3)} + \sigma^*C_v^{(2,3)} + \dots\\
& \dots + C_h^{(i-1,i)} + C_v^{(i-1,i)} + C_h^{(i)} + C_v^{(i)} + k_i L_i.
\end{align*}
Recall that additional upper indices mean the proper transforms. 

Denote the $\mult_{L_{i-1}}\mathcal{M}^{({i-1})}$ as $\lambda_i$. 
Recall that by $f_i$ we denote the class of the fibre of the ruled surface $E^{(i)}$.
Thus $C_v^{(i)} \equiv d_v^{(i)} f_i$ and $C_h^{(i)} \equiv d_h^{(i)} L_i + \beta_i f_i$ for some $d_v^{(i)}$, $d_h^{(i)}$, and $\beta_i$. 
Also $d_h^{(i)}\leqslant \beta_i$, because $C_h^{(i)}$ does not contain the exceptional section. 
Recall that $\delta = 1$ if $L_1\subset F^{(1)}$ and $\delta = 0$ otherwise. 
We now describe how the classes in components of $Z_{i}$ change as we climb up the staircase.

\begin{Lemma} \label{LemCycleRel}
We have the following relations for the proper transforms and the pullbacks of the cycles
\begin{align*}
&\Delta^{(1)} \equiv \sigma^*_1 \Delta^{(0)} - \big( \Delta^{(0)}\cdot L_0 \big)_{F^{(0)}} f_{1},\\
&\Delta^{(2)} \equiv \sigma^*_2 \Delta^{(1)} - \delta (\Delta^{(0)} \cdot L_0)_{F^{(0)}} f_{2},\\
&C_h^{(1,i+1)} \equiv \sigma^*_{i+1} C_h^{(1,i)} - \beta_1 f_{i+1} 
\quad \text{for} ~ i\geqslant 1,\\
&C_v^{(1,i+1)} \equiv \sigma^*_{i+1} C_h^{(1,i)} - d_v^{(1)} f_{i+1} 
\quad \text{for} ~ i\geqslant 1,\\
&C_h^{(i,i+1)} \equiv \sigma^*_{i+1} C_h^{(i)} - \big( \beta_i - d^{(i)}_h \big) f_{i+1} 
\quad \text{for} \quad i\geqslant 2,\\
&C_v^{(i,i+1)} \equiv  \sigma^*_{i+1} C_v^{(i)} - d^{(i)}_v f_{i+1}
\quad \text{for} \quad i\geqslant 1,\\
&Z_h^{(i+1)} \equiv \sigma^*_{i+1} Z_h^{(i)} - \alpha_{i+1} f_{i+1} \quad \text{for some} \quad \alpha_{i+1} \leqslant 2n^2.
\end{align*}
\end{Lemma}
\begin{proof}
First equivalence follows from Lemma \ref{LemCycle1} directly.
If $\delta = 0$, that is if $F^{(1)}$ and hence $\Delta^{(1)}$ is disjoint from $L_1$ then we have $\Delta^{(2)} \equiv \sigma^*_2 \Delta^{(1)}$.
If $\delta = 1$, that is if $L_1 \subset F^{(1)}$, then the equivalence follows from Lemma \ref{LemCycle1}.

Note that 
\begin{align*}
&(C^{(1,i)}_h \cdot L_i)_{E^{(1,i)}} = (C^{(1)}\cdot L_1)_{E^{(1)}} = \beta_1,\\
&(C^{(1,i)}_v \cdot L_i)_{E^{(1,i)}} = (C^{(1)}\cdot L_1)_{E^{(1)}} = d_v^{(1)},\\
&(C_h^{(i)} \cdot L_i)_{E^{(i)}} = \beta_i - d^{(i)}_h,\\
&(C_v^{(i)} \cdot L_i)_{E^{(i)}} = d^{(i)}_v.
\end{align*}
Thus all equivalences but the last follow from Lemma \ref{LemCycle1}. 
By Lemma \ref{LemCycle2} we have 
$$
\alpha_i \lem Z_h^{(i)} \cdot E^{(i)} \lem Z_h \cdot F
$$
Let $D_1,D_2$ be general divisors in $\mathcal{M}$ then
$$
Z_h \cdot F = D_1 \cdot D_2 \cdot F = 2 n^2
$$
since $D_i \in \big| -n K_X + l n F \big|$.
\end{proof}

\begin{Lemma} \label{LemDegRel}
We have the following relations for vertical degrees on $X^{(i)}$. 
For $i=1$
\begin{align*}
\beta_1 + d_v^{(1)} = \alpha_1 + \big( \Delta^{(0)} \cdot L_0 \big)_{F^{(0)}}
- \lambda_1 \big( n - 2\nu_Q(D) \big) - 2\lambda_1^2 - k_1 - d_h^{(1)},
\end{align*}
for $i=2$
\begin{align*}
\beta_2 + d_v^{(2)} = \alpha_2 + \delta \big( \Delta^{(0)} \cdot L_0 \big)_{F^{(0)}} 
- \lambda_1 \big( n - 2 \nu_Q(D) \big) - \lambda_2^2 - 2 \lambda_1 \lambda_2 + 
\beta_1 + d_v^{(1)},
\end{align*}
and for $i\geqslant3$
\begin{align*}
\beta_i + d_v^{(i)} = \alpha_{i-1} + \big( \beta_1 + d_v^{(1)} \big) 
+ d_v^{(i-1)} + \big( \beta_{i-1} - d_h^{(i-1)} \big) 
- \lambda_i \big( n - 2 \nu_Q(D) \big) - 2 \lambda_1 \lambda_i - \lambda_i^2,
\end{align*}
where $D$ is a general member in $\mathcal{M}$.
\end{Lemma}

\begin{proof}
By Lemma \ref{LemD2} we have the equivalence
\begin{align*}
Z_1 & \equiv \sigma^{*}(Z_0) + \lambda_1^2 \big( E^{(1)} \big)^2 
- 2 \lambda_1 \big( D^{(0)} \cdot L_{0} \big) f_1 \\
& \equiv \sigma^{*} \big( Z_0 \big) - \lambda_1^2 \sigma^{*}(L_0) 
- \Big( \lambda_1 \big( n - 2 \nu_Q(D) \big) + 2 \lambda_1^2 \Big) f_1.
\end{align*}
On the other hand from the decomposition of $Z_1$, Lemma \ref{LemCycleRel}, and Lemma \ref{LemStaircase2} $(i)$ we have
\begin{align*}
&Z_1 = Z_h^{(1)} + \Delta^{(1)} + C_h^{(1)} + C_v^{(1)} + k_1 L_1 \equiv\\
\equiv \sigma_1^* \big( &Z_h^{(0)} + \Delta^{(0)} + (k_1 + d_h^{(1)}) L_0 \big) + (\beta_1 + d_h^{(1)}) f_1 + C_v^{(1)} 
- \big( \alpha_1 + \big( \Delta^{(0)} \cdot L_0 \big)_{F^{(0)}} - k_1 \big)f_1.
\end{align*}
Combining these equivalences we conclude that in $A^2(X^{(1)})/\sigma_1^* A^2(X^{(0)})$ we have
\begin{align*}
(\beta_1 + d_h^{(1)} + d_v^{(1)}) f_1 \equiv C_h^{(1)} + C_v^{(1)} \equiv 
- \Big( \lambda_1 \big( n - 2 \nu_Q(D) \big) + 2 \lambda_1^2 \Big) f_1 
+ \big( \alpha_1 + \big( \Delta^{(0)} \cdot L_0 \big)_{F^{(0)}} - k_1 \big) f_1.
\end{align*}

Similarly by Lemma \ref{LemD2}
\begin{align*}
&Z_2 \equiv \sigma^{*} \big( Z_1 \big) + \lambda_2^2 \big( E^{(2)} \big)^2 
- 2 \lambda_2 \big( D^{(1)}_1 \cdot L_{1} \big) f_2 \equiv\\
&\equiv \sigma^{*} \big( Z_1 \big) - \lambda_2^2 \sigma^{*}(L_{1}) 
- \Big( \lambda_2 \big( n - 2 \nu_Q(D) \big) 
+ \lambda_2^2 + 2 \lambda_1 \lambda_2 \Big) f_2.
\end{align*}
From the decomposition of $Z_2$ and Lemma \ref{LemDegRel} we see that
\begin{align*}
Z_2 = Z_h^{(2)} 
+ \Delta^{(2)} + C_h^{(1,2)} + C_v^{(1,2)} + C_h^{(2)} & + C_v^{(2)}+ k_2 L_2 \equiv\\
 \equiv \sigma^* \big( Z_h^{(1)} + \Delta^{(1)} + C_h^{(1)} + C_v^{(1)} + k_2 L_{1} \big) &
+ C_h^{(2)} + C_v^{(2)}- \\
- \big( \alpha_{2} + \delta \Delta^{(0)} & \cdot L_0 + \beta_{1} + d^{(1)}_v \big) f_2.
\end{align*}
Combining these equalities we conclude that in $A^2(X^{(2)})/\sigma_1^* A^2(X^{(1)})$ we have
\begin{align*}
C_h^{(2)} + C_v^{(2)} \equiv (d^{(2)}_v + \beta_2)f_2 \equiv \big( \alpha_{2}& + \delta \big( \Delta^{(0)} \cdot L_0 \big)_{F^{(0)}} + \beta_1 + d_v^{(1)} \big) f_2\\
&- \Big( \lambda_2 \big( n - 2 \nu_Q(D) \big) + 2 \lambda_1 \lambda_2 + \lambda_2^2 \Big) f_2.
\end{align*}
Hence we find $d^{(2)}_v+\beta_2$.

We treat the general case the same. Once again by Lemma \ref{LemD2} we have
\begin{align*}
&Z_i \equiv \sigma^{*} \big( Z^{(i-1)} \big) + \lambda_i^2 \big( E^{(i)} \big)^2 
- 2 \lambda_i \big( D^{(i-1)}_1 \cdot L_{i-1} \big) f \equiv \\
&\equiv \sigma^{*} \big( Z^{(i-1)} \big) - \lambda_i^2 \sigma^{*} (L_{i-1}) 
- \Big( \lambda_i \big( n - 2 \nu_Q(D) \big) 
+ \lambda_i^2 + 2 \lambda_1 \lambda_i \Big) f_i.
\end{align*}
Once again from the decomposition of $Z_i$ and Lemma \ref{LemCycleRel} we see that
\begin{align*}
&Z_i \equiv Z_h^{(i)} + C_h^{(1,i)} + C_v^{(1,i)} 
+ \sigma^* \big( \Delta^{(2)} + C_h^{(2,3)} + C_v^{(2,3)} + \dots 
+ C_h^{(i-2,i-1)} + C_v^{(i-2,i-1)} \big) + \\
& + C_h^{(i-1,i)} + C_v^{(i-1,i)} + C_h^{(i)} + C_v^{(i)} + k_i L_i
= \sigma^* \big( Z_h^{(i-1)} + \Delta^{(1+\delta)} + C_h^{(1,i-1)} + C_v^{(1,i-1)} + \\
&+ C_h^{(2,3)} + C_v^{(2,3)} + \dots + C_h^{(i-2,i-1)} + C_v^{(i-2,i-1)} 
+ C_h^{(i-1)} + C_v^{(i-1)} + k_i L_{i-1} \big) + C_h^{(i)} + C_v^{(i)} -\\
&- \big( \alpha_{i} + \beta_1 + d_v^{(1)} + 
(\beta_{i-1} - d^{(i-1)}_h) + d^{(i-1)}_v \big) f_i.
\end{align*}
Combining these equalities we conclude that in $A^2(X^{(i)})/\sigma_1^* A^2(X^{(i-1)})$ we have
\begin{align*}
C^{(i)}_h + C^{(i)}_v \equiv 
\big( \alpha_{i-1} + \big( \Delta^{(0)} \cdot L_0 \big)_{F^{(0)}} + (\beta_{i-1} - d^{(i-1)}_h) + &d^{(i-1)}_v  + \beta_1 + d_v^{(1)} \big) f_i-\\
 &- \Big( \lambda_i \big(n - 2 \nu_Q(D) \big) 
+ 2 \lambda_1 \lambda_i + 2 \lambda_i^2 \Big)f_i.
\end{align*}
Hence we find $d^{(i)}_v+\beta_i$.
\end{proof}

\begin{Cor} \label{CorDegBound}
The vertical degrees are bounded as follows
\begin{align*}
&\beta_1 + d_v^{(1)} < 2 n^2 - 2 \lambda_1^2 + 4 n^2 \gamma,\\
&\beta_M + d_v^{(M)} < (M-1) \big( 2 n^2 - 2 \lambda_1^2 + 4 n^2 \gamma \big) 
+ \sum_{i=2}^{M} \big( 2 n^2 - \lambda_i^2 - 2 \lambda_1 \lambda_i \big) 
+ 4 \delta n^2 \gamma~~  \text{for}~ M \gem 2.
\end{align*}
\end{Cor}
\begin{proof}
Since the pair $(X,\frac{1}{n}\mathcal{M})$ is canonical at $Q$ the inequality 
$n \gem 2 \nu_Q(D)$ holds. 
For $i=1$ by Lemma \ref{LemDegRel} we have
\begin{align*}
\beta_1 + d_v^{(1)} = 
\alpha_1 + \big( \Delta^{(0)} \cdot L_0 \big)_{F^{(0)}} - \lambda_1 ( n - 2 \nu_Q(D) ) - 2 \lambda_1^2 - k_1 - d_h^{(1)}.
\end{align*}
Combining this with the inequalities $\alpha_1 \leqslant 2n^2$,  $k_1, d_h^{(1)} \gem 0$, and $\big( \Delta^{(0)}\cdot L_0 \big)_{F^{(0)}} \leqslant \big(\Delta \cdot L\big)_F < 4 \gamma n^2$ we get
\begin{align*}
\beta_1 + d_v^{(1)} < 2 n^2 + 4 \gamma n^2 - 2 \lambda_1^2.
\end{align*}

We prove the second bound by induction. Suppose $M=2$. Then using the same bounds we get
\begin{align*}
&\beta_2 + d_v^{(2)} = \alpha_{2} + \delta \big( \Delta^{(0)} \cdot L_0 \big)_{F^{(0)}} + d_v^{(1)} + \beta_{1}
- \lambda_i ( n - 2 \nu_Q(D) ) - \lambda_1^2 - 2 \lambda_1 \lambda_2 < \\
&< (2 n^2 + 4 \gamma n^2 - \lambda_1^2) + 
(2 n^2 - \lambda_2^2 - 2 \lambda_1\lambda_2) + 4 \delta \gamma n^2.
\end{align*}

Now suppose the inequality holds for $M-1$. 
Then similar to the previous case we have
\begin{align*}
&\beta_i + d_v^{(i)} = \alpha_{i} + (\beta_1 + d_v^{(1)}) + d_v^{(i-1)} 
+ (\beta_{i-1} - d_h^{(i-1)}) 
- \lambda_i ( n - 2 \nu_Q(D) ) - 2 \lambda_1 \lambda_i - \lambda_i^2 <\\
&< ( 2 n^2 - 2 \lambda_1 \lambda_i - \lambda_i^2 ) 
+ (\beta_1 + d_v^{(1)}) + d_v^{(i-1)} + \beta_{i-1}.
\end{align*}
Combining this with the bound on $\beta_1 + d_v^{(1)}$ and the assumption of induction we get the statement of the corollary.
\end{proof}

\begin{Cor} \label{CorMultBound}
The following bounds on multiplicities hold.
\begin{enumerate}[(i)]
\item Let $B$ be a fibre of a ruled surface $E^{(i)}$ then 
\begin{align*}
 \mult_B Z_1 =& \mult_B C^{(1)}_v 
\lem 2 n^2 - 2 \lambda_1^2 + 4 n^2 \gamma,\\
 \mult_B Z_i =& \mult_B C^{(i)}_v \lem
\sum_{i=2}^{M} \big( 4 n^2 - 2 \lambda_1^2 - \lambda_i^2 - 2 \lambda_1 \lambda_i \big)
+ 4 (M - 1 + \delta) n^2 \gamma, \quad \text{for} ~ i\geqslant 2.
\end{align*}
\item Let $B$ be a point on $E^{(i)}$ and suppose $B$ is not on the exceptional section then
\begin{align*}
&\mult_B (C^{(1)}_v + C^{(1)}_h) 
\lem \big( 2 n^2 - 2 \lambda_1^2 + 4n^2 \gamma \big)\\
&\mult_B (C^{(i)}_v + C^{(i)}_h) \lem
\sum_{i=2}^{M} \big( 4 n^2 - 2 \lambda_1^2 - \lambda_i^2 - 2 \lambda_1 \lambda_i \big)
+ 4 (M - 1 + \delta) n^2 \gamma, \quad \text{for} ~ i\geqslant 2.
\end{align*}

\end{enumerate}
\end{Cor}
\begin{proof}
Clearly $\mult_B C^{(i)}_v$ is bounded by a vertical degree $d^{(i)}_v$ whether $B$ is a point or a curve. 
Thus the inequality holds if $B$ is a curve.

Similarly, if $B$ is a point, then $\mult_B C^{(i)}_h\leqslant d^{(i)}_h$, hence $\mult_B (C^{(i)}_v + C^{(i)}_h) \leqslant d^{(i)}_v+d^{(i)}_h$. 
Since $C^{(i)}_h$ does not contain the exceptional section we have $d^{(i)}_h\leqslant \beta_i$. 
Therefore by Corollary \ref{CorDegBound} the inequalities hold.
\end{proof}

\section{Method of supermaximal singularity upstairs} \label{SMupstairs}

In the previous section we found an upper bound on the multiplicity of components of $Z_M$ at the centre of $\nu$ on $X^{(M)}$. In this section we show that it contradicts Corti's inequality.

\begin{Lemma} \label{LemLogUp}
The pair 
\begin{align*}
\bigg( X^{(M)},
\frac{1}{n} \mathcal{M}^{(M)} 
- \Big( \frac{1}{n}(n - \lambda_1 ) &+ \gamma \Big) E^{(1,M)} -\\
-&\sum_{j=2}^M \Big( \frac{1}{n} \sum_{i=2}^{j} ( 2 n - \lambda_1 -  \lambda_i) 
+ (j - 1 + \delta) \gamma \Big) E^{(j,M)} - \gamma F^{(M)} \bigg)
\end{align*}
is not canonical at $\nu$ on $X^{(M)}$.
\end{Lemma}
\begin{proof}
By Theorem \ref{ThStaircase} we have $L_i\subset E^{(1,i)}$ for all $i$ and 
$L_i \not\subset E^{(j,i)}$ for $i > j > 1$, therefore the following equivalence holds
\begin{align*}
K_{X^{(M)}} - E^{(1,M)} - \sum_{i=2}^{M} 2 (i - 1) E^{(i,M)} - \frac{1}{2} E_Q 
\sim \sigma^*(K_X).
\end{align*}
We disregard $E_Q$ in this equivalence and in the next ones since $E_Q$ is away from $P$.
For a general member $D\in\mathcal{M}$ we have 
\begin{align*}
&D^{(1)} + \lambda_1 E^{(1)} \sim \sigma^*(D) \quad \text{and} \\
&D^{(M)} + \sum_{j=1}^M \Big( \sum_{i=2}^j (\lambda_i + \lambda_1) E^{(j,M)} \Big) 
\sim \sigma^*(D) \quad \text{for} ~ j\gem 2.
\end{align*}
By Theorem \ref{ThStaircase} $(vii)$ and $(viii)$ we have 
\begin{align*}
F^{(M)} + E^{(1,M)} + \sum_{j=2}^{M} ( j - 1 + \delta ) E^{(i,M)} = \sigma^*(F).
\end{align*}
Thus the pair in the statement of the lemma is the log pullback of the pair
\begin{align*}
\Big( X, \frac{1}{n} \mathcal{M} - \gamma F \Big).
\end{align*}
Hence by Remark \ref{RemLogPull} the pair is not canonical at $\nu$.
\end{proof}

\paragraph*{\bf End of the proof of the Main Theorem}
Suppose $X^{(M)}$ is as in Proposition \ref{PropStopCond}, that is the centre of $\nu$ on $X^{(M)}$ is not a point on the exceptional section of $E^{(M)}$. 
We consider the three possibilities for the centre of $\nu$, and in each case we obtain a contradiction. 

\subsection{Case A}

Suppose the centre $B$ of $\nu$ on $X^{(M)}$ is a fibre of the ruled surface $E^{(M)}$. Then the only divisor in the boundary which contains $B$ is $E^{(M)}$. 
First suppose $M=1$, then the pair
\begin{align*}
\bigg( X^{(1)},
\frac{1}{n} \mathcal{M}^{(1)}
- \Big( \frac{1}{n}(n - \lambda_1 ) + \gamma \Big) E^{(1)} \bigg)
\end{align*}
is not canonical at $\nu$. 
By Lemma \ref{LemVanya}
\begin{align*}
\mult_B Z_1 > 4 n^2 - 4 n \lambda_1 + 4 n^2 \gamma.
\end{align*}
Combining this inequality with Corollary \ref{CorMultBound} we get
\begin{align*}
2 n^2 - 2 \lambda_1^2 + 4 n^2 \gamma > 4 n^2 \frac{n - \lambda_1}{n} + 4 n^2 \gamma
\end{align*}
or, equivalently,
\begin{align*}
0>2n^2-4n\lambda_1+2\lambda_1^2=2(n-\lambda_1)^2,
\end{align*}
which is a contradiction.

Now suppose $M \gem 2$.
By Lemma \ref{LemLogUp}
\begin{align*}
\bigg( X^{(M)},
\frac{1}{n} \mathcal{M}^{(M)} 
- \Big( \frac{1}{n} \sum_{i=2}^{M} ( 2 n - \lambda_1 -  \lambda_i)
+ (M - 1 + \delta) \gamma \Big) E^{(M)} \bigg)
\end{align*}
is not canonical at $\nu$. 
By Lemma \ref{LemVanya} we have
\begin{align*}
\mult_B Z_M > \sum_{i=2}^M ( 8 n^2  - 4 n \lambda_1 - 4 n \lambda_i )
+ 4 (M - 1 + \delta) n^2 \gamma.
\end{align*}
Combining this inequality with Corollary \ref{CorMultBound} we get
\begin{align*}
\sum_{i=2}^M (4 n^2 - 2 \lambda_1^2 - \lambda_i^2 &- 2 \lambda_1 \lambda_i) 
+ 4 (M - 1 + \delta) n^2 \gamma > \\
&> \sum_{i=2}^M ( 8 n^2  - 4 n \lambda_1 - 4 n \lambda_i ) 
+ 4 (M - 1 + \delta) n^2 \gamma.
\end{align*}
or, equivalently,
\begin{align*}
(M - 1) \lambda_1^2 + \sum_{i=2}^M \big( 2 n - \lambda_1 - \lambda_i \big)^2 < 0,
\end{align*}
which is again a contradiction.

\subsection{Case B}
Suppose the centre $B$ of $\nu$ on $X^{(M)}$ is a point which is not on $E^{(M-1,M)}$.
Then the only divisors in the boundary containing $B$ are $E^{(M)}$ and possibly $F^{(M)}$ if $M = 2$. 

First suppose $M = 2$ and $B \in F^{(2)}$, then $\delta = 1$.
Then the pair
\begin{align*}
\bigg( X^{(2)},
\frac{1}{n} \mathcal{M}^{(2)}
- \Big( \frac{1}{n}( 2 n - \lambda_1 - \lambda_2) + 2 \gamma \Big) E^{(2)}
- \gamma F^{(2)} \bigg)
\end{align*}
is not canonical at $\nu$ and the components of $Z_2$ which may contain $B$ are: $Z_h^{(2)}$, $C_v^{(2)}$, $C_h^{(2)}$, and $\Delta^{(2)}$.
By Corti's inequality there are numbers $0 < t, t_F \lem 1$ such that
$$
\mult_B Z_h^{(2)} + t \mult_B \big( C_v^{(2)} + C_h^{(2)} \big) + t_F \mult_B \Delta^{(2)}
\gem 4 n^2 + t ( 8 n^2 - 4 n \lambda_1 - 4n \lambda_2 + 8 n^2 \gamma ) + 4 t_F n^2 \gamma.
$$
On the other hand
$$
\mult_B \Delta^{(2)} \lem \Delta^{(2)} \cdot E^{(2)} = \big( \Delta^{(0)} \cdot L_0 \big)_{F^{(0)}} \leqslant 4 n^2 \gamma
$$
and we already mentioned the bounds on the other cycles.
Combining the bounds we get 
$$
2 n^2 + t ( 4 n^2 - 2 \lambda_1^2 - \lambda_2^2 - 2 \lambda_1 \lambda_2 + 8 n^2 \gamma ) + 4 t_F n^2 \gamma
> 4 n^2 + t ( 8 n^2 - 4 n \lambda_1 + 8 n^2 \gamma ) + 4 t_F n^2 \gamma,
$$
or, equivalently, 
$$
2 n^2 + t \lambda_1 ^2 + t (2 n - \lambda_1 - \lambda_i)^2 < 0,
$$
which is again a contradiction.

Note that the proper transform on $F^{(2)}$ of the half-line $L$ which gave us so much trouble earlier is a part of $C_h^{(2)}$, thus its contribution to multiplicity is bounded.

Suppose $M \gem 2$ and $B \not\in F^{(2)}$.
Then the pair
\begin{align*}
\bigg( X^{(M)},
\frac{1}{n} \mathcal{M}^{(M)} 
- \Big( \frac{1}{n} \sum_{i=2}^{M} ( 2 n - \lambda_1 - \lambda_i)
+ (M - 1 + \delta) \gamma \Big) E^{(M)} \bigg)
\end{align*}
is not canonical at $\nu$ and the components of $Z_M$ which may contain $B$ are $Z_h^{(M)}$, $C_v^{(M)}$, and $C_h^{(M)}$.
By Corti's inequality there is a number $0 < t \lem 1$ such that
$$
\mult_B Z_h^{(M)} + t \mult_B \big( C_v^{(M)} + C_h^{(M)} \big)
\gem 4 n^2 + t \sum_{j=2}^M \big( 8 n^2 - 4 \lambda_1 n - 4 \lambda_i n \big)
+ 4 t (M - 1 + \delta) n^2 \gamma.
$$
Combining this inequality with the bounds from Corollary \ref{CorMultBound} we get
\begin{align*}
 2 n^2 
+ t \sum_{j=2}^M \big( 4 n^2 - 2 \lambda_1^2 &- \lambda_i^2 - 2 \lambda_1\lambda_i  \big)
+ 4 t (M - 1 + \delta) n^2 \gamma > \\
& > 4 n^2 + t \sum_{j=2}^M \big( 8 n^2 - 4 \lambda_1 n - 4 \lambda_i n \big)
+ 4 t (M - 1 + \delta) n^2 \gamma,
\end{align*}
or, equivalently,
$$
2 n^2 + t (M-1) \lambda_1^2 + \sum_{i=2}^M (2 n - \lambda_1 - \lambda_i)^2 < 0,
$$
contradiction.

\subsection{Case C}
Suppose the centre $B$ of $\nu$ on $X^{(M)}$ is a point on the intersection $E^{(M)}\cap E^{(M-1,M)}$. 
These are the only divisors of the boundary containing $B$, that is $B \not\in F^{(M)}$.
Denote $M^- = M - 1$ for compactness of formulae.
The components of $Z_M$ which may contain $B$ are:
$Z_h^{(M)}$, $C_v^{(1,M)}$, $C_h^{(1,M)}$, $C^{(M^-,M)}_h$, and $C^{(M^-,M)}_v$.
Also the pair 
\begin{align*}
\bigg( X^{(M)},
\frac{1}{n} \mathcal{M}^{(M)} 
-  \Big( \frac{1}{n} \sum_{i=2}^{M^-} ( 2 n - \lambda_1 - \lambda_i)&
+ (M^- - 1 + \delta) \gamma \Big) E^{(M^-,M)} - \\
& - \Big( \frac{1}{n} \sum_{i=2}^{M} ( 2 n - \lambda_1 - \lambda_i)
+ (M - 1 + \delta) \gamma \Big) E^{(M)} \bigg)
\end{align*}
is not canonical at $\nu$.
By Corti's inequality there are numbers $0<t,t^-\leqslant 1$ such that
\begin{align*}
\mult_B Z_h^{(M)} + t \mult_B \big( C_v^{(M)}& + C_h^{(M)} \big) 
+ t_- \mult_B \big( C_v^{(M^-,M)} + C_h^{(M^-,M)} \big) \gem \\
\gem 4 n^2 + t \sum_{j=2}^M \big( 8 n^2 - &4 \lambda_1 n - 4 \lambda_i n \big)
+ t^- \sum_{j=2}^{M^-} \big( 8 n^2 - 4 \lambda_1 n - 4 \lambda_i n \big) +\\
&+ 4 t (M - 1 + \delta) n^2 \gamma + 4 t^- (M^- - 1 + \delta) n^2 \gamma.
\end{align*}
By combining the inequality with the bounds from Corollary \ref{CorMultBound} we get
\begin{align*}
& 2 n^2 
+ t \sum_{j=2}^M \big( 4 n^2 - 2 \lambda_1^2 - \lambda_i^2 - 2 \lambda_1\lambda_i  \big)
+ t^- \sum_{j=2}^{M^-} \big( 4 n^2 
- 2 \lambda_1^2 - \lambda_i^2 - 2 \lambda_1\lambda_i  \big) + \\
& + 4 t (M - 1 + \delta) n^2 \gamma + 4 t^- (M^- - 1 + \delta) n^2 \gamma > \\
& > 4 n^2 + t \sum_{j=2}^M \big( 8 n^2 - 4 \lambda_1 n - 4 \lambda_i n \big)
+ t^- \sum_{j=2}^{M^-} \big( 8 n^2 - 4 \lambda_1 n - 4 \lambda_i n \big) + \\
&+ 4 t (M - 1 + \delta) n^2 \gamma + 4 t^- (M^- - 1 + \delta) n^2 \gamma,
\end{align*}
or, equivalently,
$$
2 n^2 + t (M - 1) \lambda_1^2 + t \sum_{i=2}^M (2 n - \lambda_1 - \lambda_i)^2
+ t^- (M^- - 1) \lambda_1^2 + t^- \sum_{i=2}^{M^-} (2 n - \lambda_1 - \lambda_i)^2 < 0,
$$
which is a contradiction. 

And this ends the process of exclusion of points in fibres containing singular points of $X$, and therefore the proof of the Main Theorem is complete.\hfill$\square$

\appendix
\section{Explicit computations on $K^2$-condition}\label{K2cond}

\subsection{On preservation of the $K^2$-condition}
Here we prove a result in favour of $K^2$-condition being preserved in the running of the Sarkisov links at singular points of $X$.

Throughout this section we assume that $\pi: X \to \MP^1$ is a del Pezzo fibration and that $X$ is a general quasi-smooth hypersurface of degree $(4,\ell)$ in a $\MP(1,1,1,2)$-bundle $T = \MP(1,1,1,2)_{(0,a,b,c)}$.
We may assume $0 \lem a \lem b$ without loss of generality, and allow $c$ to take value in $\mathbb{Z}$.
Let $u,v$ be the coordinates on the base $\MP^1$ and let $x,y,z,w$ be the coordinates on the fibre.

Denote by $M_T$ the divisor with the equation $\{x = 0\}$ and let $M = M_T \big\vert_X$.
Denote by $F$ a fibre of $\pi$ on $X$ and let $F_T$ be a fibre of the $\MP(1,1,1,2)$-bundle.
Thus we have that $X \sim 4 M_T + \ell F_T$.
Clearly, for the cone of effective divisors on the toric level we have that
\[ \begin{array}{ll}
\ON{Eff}^1(T) = \langle M_T, F_T \rangle&\text{if } c\geq 0,\\
\end{array}\]
By the choice of $T$ we have $K_T = -5 M_T - (2 + a + b + c) F_T$ and by adjunction 
$$K_X = - M + (\ell - 2 - a - b - c) F.$$
The following relations for the intersection numbers are rather easy to calculate.

\begin{Lemma} \label{ToricInt}
With notation as above, we have that
\begin{align*}
& M_T^3 \cdot F_T = 1/2, \\
& 2 M_T^4 = - a - b - c/2,\\
& M^2 \cdot F = M_T^2 \cdot F_T \cdot X = 2,\\
& M^3 = M_T^3 \cdot X = \ell/2 - 2 a - 2 b - c.
\end{align*}
\end{Lemma}
\begin{proof}
The first equality is just computation on $\MP(1,1,1,2)$.
The second equality follows from the fact that the intersection 
$\{x=0\} \cap \{y=0\} \cap \{z=0\} \cap \{w=0\}$ is empty.
The rest is computed by intersection with $X$.
\end{proof}

\begin{Prop} \label{keeping-scroll}
Let $\pi\colon X \to \MP^1$ be a del Pezzo fibration of degree $2$.
Suppose $X$ is a quasi-smooth hypersurface of bi-degree $(4,\ell)$ in a $\MP(1,1,1,2)$-bundle $T$ over $\MP^1$.
Suppose $c \gem 2b$, then for any $I$ the variety $X_I$ is also a hypersurface in $T$ of bi-degree $(4,\ell)$.
\end{Prop}
\begin{proof}
It is enough to prove the statement of the proposition for $X^\prime = X_{\{i\}}$, then the general case follows for any $I$ by induction.

In Subsection~\ref{square-models} we represented $X^\prime$ as a complete intersection in $V^\prime$ a $\MP(1,1,1,2,2)$-bundle over $\MP^1$.
We change the coordinates in $V^\prime$ to have a more convenient first equation.

Suppose the singularity for which we constructed a Sarkisov link is given by $x=y=z=u=0$.
Then we can decompose the coefficient at $w^2$ in the following way
$$
f(u,v) = u v^{N-1} + u^2 g(u,v).
$$
The equations of $X^\prime$ in $V^\prime$ are
\begin{align*}
u s =  (v^{N-1} + u g(u,v)) w + q(x,y,z;u,v)
\quad\text{and}\quad
s w + r(x,y,z;u,v) = 0
\end{align*}
Let us change the coordinates $s = s_{\text{new}} + g(u,v) w$.
Then the new equations are
\begin{align*}
u s =  v^{N-1} w + q(x,y,z;u,v),
\quad\text{and}\quad
s w + g(u,v) w^2 + r(x,y,z;u,v) = 0.
\end{align*}
Since $c \gem 2b$ we may also decompose 
$$
q(x,y,z;u,v) = u q_1(x,y,z;u,v) + v^{N-1} q_2(x,y,z;u,v).
$$
Indeed, the degree of the first equation is $(2,c+N)$, thus
$$
q \in \langle u,v \rangle^{c+N-2b} \subset \langle u,v \rangle^N,
$$
where $\langle u,v \rangle$ is the ideal generated by $u,v$.
Then we change coordinates 
\begin{align*}
&s = s_{\text{new}} + q_1(x,y,z;u,v),\\
&w = w_{\text{new}} - q_2(x,y,z;u,v).
\end{align*}
The new first equation is $u s =  v^{N-1} w$.
Clearly this equation defines a $\MP(1,1,1,2)$-bundle over $\MP^1$.
Let $(u_T,v_T,x_T,y_Y,z_T,w_T)$ be the coordinates on $T$.
Consider a map from $T$ into $V^\prime$ defined as follows
\begin{align*}
&u = u_T, \quad v = v_T, \quad x = x_T, \quad y = y_T, \quad z = z_T,\\
&w = u w_T, \quad s = v^{N-1} w_T.
\end{align*}
It is easy to see that this map is well defined and gives an embedding of $T$ into $V^\prime$.
The equation of $T$ in $V^\prime$ is $u s = v^{N-1} w$, that is the first equation of $X^\prime$.
Thus, we have $X^\prime$ embedded into $T$, substituting coordinates of $T$ into the second equation we conclude that the equation of $X^\prime$ has degree $(4,\ell)$ that is the same degree that the equation of $X$ in $T$.
\end{proof}

\begin{Remark}
Let $X$ be a in Proposition \ref{keeping-scroll} and suppose $X$ satisfies the $K^2$-condition.
Then all $X_I$ live in the same family of hypersurfaces by Proposition \ref{keeping-scroll}.
Yet, we cannot say whether $X_I$ also satisfies the $K^2$-condition.

The general expectation is that $K^2$-condition is an open property in the moduli. Given this, let $U$ be an open subset in a given family for which the $K^2$-condition is satisfied. Then by Proposition \ref{keeping-scroll}, $X$ satisfies the total $K^2$-condition if and only if $X_I \in U$ for all $I \subset \{1,\dots,N\}$, which suggests that $K^2$-total condition must also be open in moduli.
\end{Remark}

\subsection{Sufficient conditions for $K^2$-total condition}

\begin{Lemma} \label{sameint}
Let $D_1, D_2, D_3$ be divisors on $X$. Fix $I \subset \{1,\dots,N\}$ and let $G_1, G_2, G_3$ be divisors on $X_I$ of the same bi-degrees.
Then $D_1 \cdot D_2 \cdot D_3 = G_1 \cdot G_2 \cdot G_3$.
\end{Lemma}
\begin{proof}
It is enough to check it for the case of $\deg G_1 = \deg G_2 = (1,0)$ and $\deg G_3 = (1,0)$ or $\deg G_3 = (0,1)$.
We are going to compute the intersections on the ambient scroll $V_{\big| I \big|}$.
Let $M_{\big| I \big|}$ be the divisor given by $x=0$ and let $F_{\big| I \big|}$ be a fibre of the projection onto $\MP^1_{u,v}$.
Then similar to Lemma \ref{ToricInt}, we compute that $M_{\big| I \big|}^4 \cdot F_{\big| I \big|} = 1/4$ and 
$4 M_I^5 = - a - b - \ell/2$.
Now suppose $\deg G_3 = (0,1)$, then
$$
G_1 \cdot G_2 \cdot G_3 = M_{\big| I \big|}^2 \cdot F_{\big| I \big|} \cdot (2,\ell-c) \cdot (4,\ell) = 2 .
$$
If instead $\deg G_3 = (1,0)$, then
$$
G_1 \cdot G_2 \cdot G_3 = M_{\big| I \big|}^3\cdot (2,\ell-c) \cdot (4,\ell) = \ell/2 - 2a - 2b - c.
$$
These intersections coincide with the ones we computed on $X$ in Lemma \ref{ToricInt}
\end{proof}

\begin{Lemma} \label{defnef}
For $X_I$ for any $I$ we have the following.
Let $M_I$ denote the divisor given by $x=0$ and let $F_I$ be a fibre of $\pi_I$.
Then the divisor $M_I + b F_I$ is nef.
\end{Lemma}
\begin{proof}
The base locus of $\big| M_I + b F_I \big|$ is contained in $\{x=y=z=0\}$.
On the other hand
$$
\{x=y=z=0\} = \{ \prod_{i\in I} l_i = w = 0 \} \cup \{ \prod_{j\not\in I} l_j = s = 0 \} = \ON{Sing} X_I,
$$
which is of cardinality $N$.
Thus the base locus of $\big| M_I + b F_I \big|$ is finite hence $M_I + b F_I$ is nef.
\end{proof}

\begin{proof}[Proof of Proposition~\ref{PrK2}]
First, let us compute $K_{X_I}$.
By adjunction we have
\begin{align*}
\deg K_{X_I} = \deg K_{V_{\big| I \big|}} + (2,c+N) + &(4,2c+N) = \\
&=\Big(-1, 3c + 2N - 2 
- a - b  - \big( c + \big|I\big| \big) - (c + N - \big| I \big|) \Big).
\end{align*}
Since $d = 2c + N$ we have
$$
\deg K_{X_I} = (-1, \ell - 2 - a - b - c) = \deg K_X.
$$
If the intersection of $K_{X_I}^2$ with a nef divisor is non positive, then $K_{X_I}^2$ is not in the interior of the Mori cone.
By Lemma \ref{defnef} the divisor $D_I = M_I + b F_I$ is nef and by Lemma \ref{sameint} the intersection of $K_{X_I}^2$ with $D_I$ is the same on all models $X_I$, thus it is enough to compute it on $X$.

We compute
\begin{align*}
K_X^2 \cdot D = (-1, \ell - a - b - c - 2)^2 \cdot (1,b) 
&= M^3 + (4 + 2a + 2b + 2c - 2 \ell + b) F \cdot M^2 =\\
&= 8 + 2a + 4b + 3c - \frac{7}{2} \ell.
\end{align*}
If this value is non-positive all $X_I$ satisfy the $K^2$-condition.
This finishes the proof.
\end{proof}

\begin{proof}[Proof of Corollary~\ref{K^2-families}]
By assumption $c \gem 2b$ thus Proposition \ref{PrK2} and the bounds $0\lem 4a \lem 4b \lem 2c \lem \ell$ imply that
\begin{align*}
a \lem 4, \quad 2b\lem 4 + a, \quad, \text{and}~2c\lem 4 + a + 2b.
\end{align*}
Using these inequalities we fill the table.
\end{proof}


\providecommand{\bysame}{\leavevmode\hbox to3em{\hrulefill}\thinspace}
\providecommand{\MR}{\relax\ifhmode\unskip\space\fi MR }
\providecommand{\MRhref}[2]{%
  \href{http://www.ams.org/mathscinet-getitem?mr=#1}{#2}
}
\providecommand{\href}[2]{#2}

\end{document}